\documentclass[a4paper, 11 pt, twoside]{amsart}
\usepackage{srollens-en}[2011/03/03]
\usepackage[left = 3.5cm, right = 3.5cm, headsep = 6mm,
footskip = 10mm, top = 35mm, bottom = 35mm, footnotesep=5mm, headheight =
2cm]{geometry}

\usepackage[usenames, dvipsnames]{xcolor}

\usepackage[T1]{fontenc}
\usepackage{tikz, tikz-cd}
\usetikzlibrary{positioning, intersections}

\usepackage{enumitem}

\usepackage[unicode,bookmarks, pdftex, backref = page]{hyperref}
\hypersetup{colorlinks=true,citecolor=NavyBlue,linkcolor=NavyBlue,urlcolor=Orange, pdfpagemode=UseNone, breaklinks=true}

{\theoremstyle{Lehn-up}

}

\renewcommand{\epsilon}{\varepsilon}
\renewcommand{\phi}{\varphi}
\newcommand{\Diff}{\mathrm{Diff}}
\newcommand{\sing}[1]{{{#1}_{\mathrm{sing}}}}

\renewcommand{\tilde}{\widetilde}

\title{Geography of Gorenstein stable log surfaces}
\author{Wenfei Liu}
\address{Wenfei Liu \\Institut f\"ur algebraische Geometrie\\  Gottfried Wilhelm Leibniz Universit\"at Hannover\\Welfengarten 1\\ 30167 Hannover\\Germany}
\email{wliu@math.uni-hannover.de}

\author{S\"onke Rollenske}
\address{S\"onke Rollenske\\Fakult\"at f\"ur Mathematik\\Universt\"at Bielefeld\\Universit\"atsstr. 25\\33615 Bielefeld\\Germany}
\email{rollenske@math.uni-bielefeld.de}

\begin{document}
\begin{abstract}
We study the geography of Gorenstein stable log surfaces and prove two inequalities for their invariants: the stable Noether inequality and the $P_2$-inequality.

By constructing examples we show that all  invariants are realised except possibly some cases where the inequalities become equalities.

\end{abstract}

\maketitle
\setcounter{tocdepth}{1}
\tableofcontents

\section{Introduction}
The classification of algebraic surfaces has been a subject of interest in algebraic geometry ever since the foundational work of the Italian school at the beginning of last century and the complexity of this endeavour led Castelnuovo and Enriques to their saying: ``If curves have been made by God, then surfaces are the devil's mischief.''
For surfaces of general type, one aspect is their \emph{geography}, that is, the question about general restrictions on their invariants and the construction of surfaces realising all possible invariants.

The compactification $\overline{\gothM}_{a,b}$ of Gieseker's moduli space of canonical models of surfaces of general type with $a=K_X^2$ and $b= \chi(\ko_X)$ parametrises stable surfaces (see Definition~\ref{defin: slc}). That such surfaces  should be the correct higher-dimensional analogue of stable curves was first suggested  by  Koll\'ar and Shepherd-Barron \cite{ksb88}; formidable technical obstacle delayed the actual construction of the moduli space for several decades~\cite{KollarModuli}.
While it is still true that the Gieseker moduli space ${\gothM}_{a,b}$ is an open subset of $\overline{\gothM}_{a,b}$, the complement is no longer a divisor as in the moduli space of stable curves: there can be additional irreducible components and for some invariants ${\gothM}_{a,b}$ might be empty while  $\overline{\gothM}_{a,b}$ is not. This simply means that the invariants of some stable surfaces cannot be realised by surfaces of general type. 

It is actually  quite natural  to consider  also stable log surfaces (Definition~\ref{defin: slc}), where we allow a reduced boundary.  As a first step beyond the classical case we prove  two fundamental inequalities for Gorenstein stable log surfaces.

\begin{custom}[$P_2$-inequality (Theorem \ref{thm: P_2})]

Let $(X,\Delta)$ be a connected Gorenstein stable log surface. Then 
\[ \chi(X,\omega_X(\Delta))=\chi(\ko_X(-\Delta)) \geq -(K_X+\Delta)^2,\]
and equality holds if and only if $\Delta=0$ and $P_2(X)= h^0(X, \omega_X^{\tensor 2})=0$.
\end{custom}

\begin{custom}[Stable log Noether inequality (Theorem \ref{thm: log noether}, Corollary~\ref{cor: log noether})]
 Let $(X,\Delta)$ be a connected Gorenstein stable log surface. Then
\begin{gather*}
p_g(X, \Delta) = h^0(X, \omega_X(\Delta)) \leq (K_X+\Delta)^2 +2,\\
\chi(X,\omega_X(\Delta)) \leq (K_X+\Delta)^2 +2,
\end{gather*}
and the first inequality is strict if $\Delta=0$.
\end{custom}
In both cases, the strategy is to use well known results in the normal respectively smooth case. For the $P_2$-inequality this is relatively straightforward (apart from a small issue with adjunction) while for the Noether inequality one needs to control the combinatorics of the glueing process carefully.

In contrast to the case of minimal surfaces of general type, most of the possible  invariants are realised by a simple combinatorial construction explained in Section \ref{sect: big example}. All these examples are locally smoothable but global smoothability may or may not occur (see Section \ref{sect: smoothability}).

To put these results in context, let us discuss the known restrictions on invariants for some classes of surfaces with empty boundary. In the following, $X$ will always denote a surface of the specified type. In all cases $K_X$ is an ample $\IQ$-Cartier divisor so we have $K_X^2 >0$, which may however be a rational number if $X$ is not Gorenstein. 

\begin{description}[leftmargin=1cm, labelindent=.5cm]
 \item[Minimal surfaces of general type] The following well-known inequalities are satisfied:
\begin{description}[leftmargin=*, labelindent=.5cm]
 \item[{Euler characteristic}] $\chi(\ko_X)>0$.
\item[Noether inequality] $p_g(X)\leq \frac 1 2 K_X^2+2$ (or $K_X^2\geq 2\chi(\ko_X)-6$).
\item[Bogomolov--Miyaoka--Yau inequality]  $K_X^2\leq 9\chi(\ko_X)$.
\end{description}
A proof of these inequalities and references showing that almost all possible invariants are known to be realised can be found in  \cite[Ch.\ VII]{BHPV}.
\item[Normal stable surfaces]
It has been proved by Blache that also in this case $\chi(\ko_X)>0$ \cite[Thm.~2]{bla94}. There is an analogue of the Bogomolov-Miyaoka-Yau inequality (see \cite{langer03} and references therein) that can be stated in terms of the orbifold Euler-characteristic
\[ K_X^2\leq 3 e_{\mathrm{orb}}(X).\]
However,  the orbifold Euler-characteristic is not invariant under deformation so it is less suited to the moduli point of view. We show that both the classical Noether-inequality and the classical Bogomolov--Miyaoka--Yau inequality fail for normal Gorenstein stable surfaces  in Section \ref{sect: exam normal}. 
\item[General case] It is known that 
\[\left\{K_X^2\mid X\text{ stable surface}\right\}\]
is a DCC set, bounded below by $1/1726$ \cite{AM04, kollar94} but our understanding is far from complete. For example, it is very difficult to bound the index for surfaces with fixed invariants.

We expect that a kind of Noether-inequality holds also in this case, see Remark \ref{rem: noether sharp}.
\end{description}
Our results for Gorenstein stable surfaces without boundary are illustrated in Figure \ref{fig: geography}, where we also mark the points where an explicit example has been constructed. 
\begin{figure}[ht]
\scriptsize
\begin{tikzpicture}
[scale=.3, 
axes/.style={},
classical/.style={thick, green!50!white},
stable/.style={thick},
pt/.style={circle,draw, fill=black, size=1mm}
]

\begin{scope}[stable]
 \fill[orange!30!white] (-18, 18) -- (-1,1) -- (3,1) -- (20, 18) -- cycle;

\begin{scope}[classical]
\fill[green!20!white] (12, 18) -- (4, 2) -- (3,1) -- (1,1) -- (1,9) -- (2, 18) -- cycle;
\draw (3, 1) -- (4,2) -- (12, 18) ;
\draw (1, 9) -- (2,18);

 \draw (3, 1)--(1,1) -- (1,9); 
\end{scope}
 \draw (12, 18)  node[above] {Noether};
\draw (2,18) node[above] {BMY};

 \draw (-1, 1)--(3,1);
\draw (-1, 1) -- (-18,18)  node[left] {$P_2=0$};
\draw (3, 1) -- (20,18) node[above] {stable Noether};

\end{scope}

\begin{scope}[axes]
\draw[->]  (-19,0)--(19,0) node[below] {$\chi(\ko_X)$};
\draw[->]  (0,-1)--(0,19) node[left]  {$K_X^2$};
\end{scope}

\foreach \y in {1,...,18}{
\foreach \x in {-18,...,20}{
\ifnum -\x<\y
{
 \ifnum \numexpr \x-\y<2
\filldraw[blue] (\x, \y) circle (1mm);
 \fi
}
\fi
}}
\filldraw[blue] (3,1) circle (1.2mm);
\filldraw[blue] (4,2) circle (1.2mm);
\node at (-1,1) {$\emptyset$};

\node at (0,-2) [rectangle, rounded corners, fill=green!20!white] {minimal surfaces};
\node at (-12,-2) [rectangle, rounded corners, fill=orange!30!white] { Gorenstein stable  surfaces};
\filldraw[blue]    (6,-2) circle (1.2mm) node  [right, black] 
{Constructed in Sect.\ \ref{sect: big example}};
\end{tikzpicture}
 \caption{The geography of minimal and Gorenstein stable surfaces}\label{fig: geography}
\end{figure}
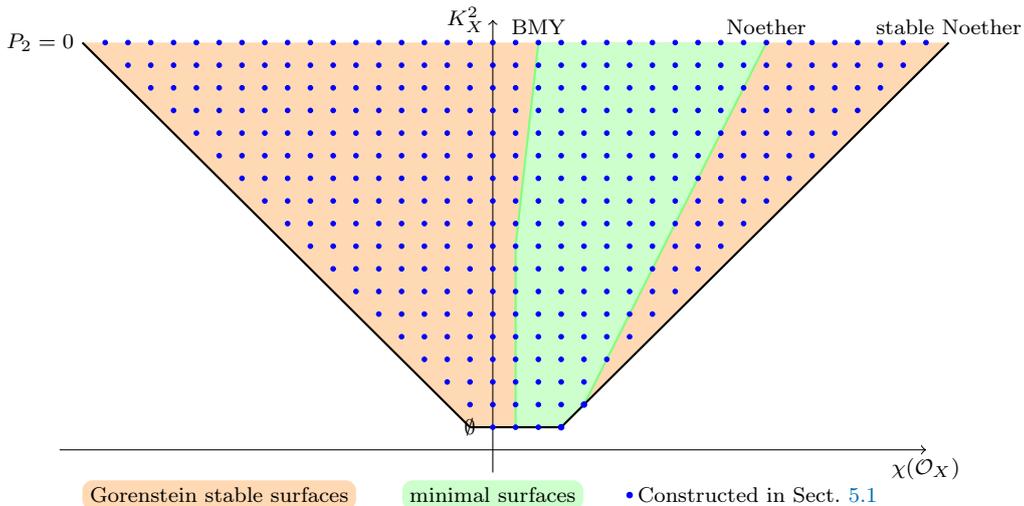

The stable log Noether inequality is sharp and we give a partial characterisation of surfaces on the stable log Noether line in Corollary \ref{cor: nonnormal equality}. On the other hand, we have some evidence to believe that there are no Gorenstein stable surfaces with $\chi(\ko_X)-2=K_X^2\geq 3$.

Surfaces with  negative $\chi(\ko_X)$ are a bit more mysterious. It can be shown that there is no surface with $K_X^2=1$ and $P_2(X)=0$ \cite{fpr13} so one might wonder if always $P_2(X)>0$.

If both the above speculations on the sharpness of the inequalities turn out to be true then every 
possible invariant is realised by the examples in Section \ref{sect: big example}.

In the last section we give some further examples that illustrate some of the obstacles in working with stable surfaces. For example, the classical approach to prove Noether's inequality would be to look at the image of the canonical map. We show that for a stable surface this image may be disconnected or not equidimensional which made a different approach necessary. We also include an example of a 1-dimensional family of stable surfaces with $\chi(\ko_X)=1$ and $K_X^2=9$, fake fake projective planes,  which confirms that there is in general no direct relation between stable surfaces and minimal surfaces with the same invariants.

\subsection*{Acknowledgements:} We are grateful to Fabrizio Catanese, Marco Franciosi, Christian Liedtke, Michael L\"onne and Rita Pardini for interesting discussions about this project. J\'anos Koll\'ar sent us a preliminary version of \cite{KollarSMMP}.  Matthias Sch\"utt suggested to construct a family of fake fake projective planes. 

Both authors were supported by DFG via the second author's Emmy-Noether project and partially via SFB 701.  The first author was supported  by the Bielefelder Nachwuchsfonds.

\subsection{Notations and conventions}
We work exclusively with schemes of finite type over the complex numbers.
\begin{itemize}
\item The singular locus of a scheme $X$ will be denoted by $\sing X$.
\item A surface is a reduced, projective scheme  of pure dimension two but not necessarily irreducible or connected.

\item A curve is a purely 1-dimensional scheme that  is Cohen--Macaulay. A curve is not assumed to be reduced, irreducible or connected; its arithmetic genus is  $p_a(C) = 1-\chi(\ko_C)$. 
\item By abuse of notation we sometimes do not distinguish a divisor $D$ and the associated divisorial sheaf $\ko_X(D)$; this is especially harmless for Cartier divisors.
\end{itemize}

\section{Preliminaries}
In this section we recall some necessary notions as well as constructions that we need throughout the text. Most of these are available in all dimensions, but for our purpose it suffices to focus on the case of surfaces. Our main reference is \cite[Sect.~5.1--5.3]{KollarSMMP}.

\subsection{Stable log surfaces}
Let $X$ be a demi-normal surface, that is,  $X$ satisfies $S_2$ and  at each point of codimension one $X$ is either regular or has an ordinary double point.
We denote by  $\pi\colon \bar X \to X$ the normalisation of $X$. The conductor ideal
$ \shom_{\ko_X}(\pi_*\ko_{\bar X}, \ko_X)$
is an ideal sheaf in both $\ko_X$ and $\ko_{\bar X} $ and as such defines subschemes
$D\subset X \text{ and } \bar D\subset \bar X,$
both reduced and of pure codimension 1; we often refer to $D$ as the non-normal locus of $X$. 

Let $\Delta$ be a reduced curve on $X$ whose support does not contain any irreducible component of $D$. Then the strict transform $\bar \Delta$ in the normalisation is well defined.
\begin{defin}\label{defin: slc}
We call a pair $(X, \Delta)$ as above a  \emph{log surface}; $\Delta$ is called the (reduced) boundary.\footnote{In general one can allow rational coefficients in $\Delta$, but we will not use this here.}

A log surface $(X,\Delta)$ is said to have \emph{semi-log-canonical (slc)}  singularities if it satisfies the following conditions: 
\begin{enumerate}
 \item $K_X + \Delta$ is $\IQ$-Cartier, that is, $m(K_X+\Delta)$ is Cartier for some $m\in\IZ^{>0}$; the minimal such $m$ is called the (global) index of $(X,\Delta)$.
\item The pair $(\bar X, \bar D+\bar \Delta)$ has log-canonical singularities. 
\end{enumerate}
The pair $(X,\Delta)$ is called stable log surface if in addition $K_X+\Delta$ is ample. A stable surface is a stable log surface with empty boundary.

By abuse of notation we say $(X, \Delta)$ is a Gorenstein stable log surface if  the index is equal to one, i.e., $K_X+\Delta$ is an ample Cartier divisor.
\end{defin}

 Since $X$ has at most double points in codimension one the map $\pi\colon \bar D \to D$ on the conductor divisors is generically a double cover and thus  induce a rational involution on $\bar D$. Normalising the conductor loci we get an honest involution $\tau\colon \bar D^\nu\to \bar D^\nu$ such that $D^\nu = \bar D^\nu/\tau$.

To state the next result we need the notion of \emph{different}, which is the correction term in the adjunction formula on a log surface.
\begin{defin}[{\cite[Definition~4.2]{KollarSMMP}}]\label{def: different}
Let $(X, \Delta)$ be a log surface and $B$ a reduced curve on $X$ that does not contain any irreducible component of the non-normal locus $D$. Suppose $\omega_X(\Delta + B)^{[m]}$ is a line bundle for some positive integer $m$. Then, denoting by $B^\nu$  the normalisation of $B$, the different $\Diff_{B^\nu}(\Delta)$ is the uniquely determined $\IQ$-divisor on $B^\nu$ such that $m\Diff_{B^\nu}(\Delta)$ is integral and the residue map induces an isomorphism
\[
 \omega_X(\Delta + B)^{[m]}\restr{B^\nu}\isom \omega_{B^\nu}^{[m]}(m\Diff_{B^\nu}(\Delta)).
\]
\end{defin}

\begin{theo}[{\cite[Thm.~5.13]{KollarSMMP}}]\label{thm: triple}
Associating to a log-surface $(X, \Delta)$ the triple $(\bar X, \bar D+\bar \Delta, \tau\colon \bar D^\nu\to \bar D^\nu)$ induces a one-to-one correspondence
 \[
  \left\{ \text{\begin{minipage}{.12\textwidth}
 \begin{center}
         stable log surfaces  $(X, \Delta)$
 \end{center}
         \end{minipage}}
 \right\} \leftrightarrow
 \left\{ (\bar X, \bar D, \tau)\left|\,\text{\begin{minipage}{.37\textwidth}
   $(\bar X, \bar D+\bar \Delta)$ log-canonical pair with 
  $K_{\bar X}+\bar D+\bar \Delta$ ample, \\
   $\tau\colon \bar D^\nu\to \bar D^\nu$  an involution s.th.\
    $\Diff_{\bar D^\nu}(\Delta)$ is $\tau$-invariant.
            \end{minipage}}\right.
 \right\}.
 \]
 \end{theo}

An important consequence, which allows to understand the geometry of stable log surfaces from the normalisation, is that
\begin{equation}\label{diagr: pushout}
\begin{tikzcd}
    \bar X \dar{\pi}\rar[hookleftarrow] & \bar D\dar{\pi} & \bar D^\nu \lar[swap]{\nu}\dar{/\tau}%[swap]{\pi}
    \\
X\rar[hookleftarrow] &D &D^\nu\lar
    \end{tikzcd}
\end{equation}
is a pushout diagram.

\begin{defin}
 Let $(X, \Delta)$ be a stable log surface. We call
\[ p_g(X, \Delta) = h^0(X, \omega_X(\Delta)) = h^2(X, \ko_X(-\Delta))\]
the geometric genus of $(X, \Delta)$ and
\[ q(X, \Delta) = h^1(X, \omega_X(\Delta)) = h^1(X, \ko_X(-\Delta))\]
the irregularity of $(X, \Delta)$. If $\Delta$ is empty we omit it from the notation.
\end{defin}
Note that in both cases for the second equality we have used \cite[Lem.~3.3]{liu-rollenske12} and duality.

We will want to relate the invariants of a stable log surface with the invariants of the normalisation.
\begin{prop}\label{prop: invariants}
 Let $(X,\Delta)$ be a stable log surface with normalisation $(\bar X,\bar \Delta)$. Then
 $(K_X+\Delta)^2 = (K_{\bar X}+\bar D+\bar \Delta)^2$  and  $\chi(\ko_X) = \chi(\ko_{\bar X})+\chi(\ko_D)-\chi(\ko_{\bar D})$.
\end{prop}
\begin{proof}
 The first part is clear. For the second note that the conductor ideal defines $\bar D$ on $\bar X$ and the non-normal locus $D$ on $X$. In particular, $\pi_*\ko_{\bar X}(-\bar D)=\ki_D$ and    additivity of the Euler characteristic for the two sequences
\begin{gather*}
 0\to \ko_{\bar X}(-\bar D)\to \ko_{\bar X}\to \ko_{\bar D}\to 0,\\
0\to \pi_*\ko_{\bar X}(-\bar D)\to \ko_{ X}\to \ko_{D}\to 0
\end{gather*}
gives the claimed result.
\end{proof}

To compare the irregularity and the geometric genus is more subtle. We state the following general result in the Gorenstein case for simplicity.
\begin{prop}[\protect{\cite[Prop.~5.8]{KollarSMMP}}]\label{prop: descend section}
If $(X, \Delta)$ is a Gorenstein log surface then 
 $\pi^*H^0(X, \omega_X(\Delta))\subset H^0(\bar X, \omega_{\bar X}(\bar D+\bar \Delta))$ is the subspace of those sections $s$  such that the residue of $s$ in $H^0(\bar D^\nu, \omega_{\bar D^\nu}(\Diff_{\bar D^\nu}(\Delta)))$  is $\tau$-anti-invariant.
 \end{prop}
\begin{rem}\label{rem: pushout sections}
We need the following consequence of the above proposition.
 Assume that $X=X_1\cup X_2$ is the union of two (not necessarily irreducible) surfaces such that the conductor is a smooth curve and let $C=X_1\cap X_2$. With $\Delta_i = \Delta\restr X_i$ the  residue maps define homomorphisms 
\[r_{X_i}\colon  H^0(X_i, K_{X_i}+\Delta_i+C) \to H^0(C, (K_{X_i}+\Delta_i+C)\restr C) \isom H^0(C, K_{C}+\Diff_C(\Delta)).\]
Then there is a fibre product diagram of vector spaces
\[ \begin{tikzcd}
    H^0(X, K_X+\Delta) \rar\dar & H^0(X_1, K_{X_1}+\Delta_1+C)\dar{r_{X_1}}\\
H^0(X_2, K_{X_2}+\Delta_2+C)\rar{-r_{X_2}}& H^0(C, K_{C}+\Diff_C(\Delta))
   \end{tikzcd}.
\]
By abuse of notation we also write
\[
 H^0(X, K_X+\Delta) = H^0(X_1, K_{X_1}+\Delta_1+C) \times_C H^0(X_2, K_{X_2} + \Delta_2+C).
\]
\end{rem}

\subsection{Gorenstein slc singularities and semi-resolutions}
Normalising a demi-normal surface looses all information on the glueing in codimension one. Often it is better to work on a simpler but still non-normal surface.
\begin{defin}
A surface $X$ is called semi-smooth if every point of $X$ is either smooth or double normal crossing or a pinch point\footnote{A local model for the pinch point in $\IA^3$ is given by the equation $x^2+yz^2=0$.}. 

A morphism of demi-normal surfaces $f\colon Y\rightarrow X$ is called a semi-resolution if the following conditions are satisfied:
 \begin{enumerate}
  \item $Y$ is semi-smooth;
  \item $f$ is an isomorphism over the semi-smooth open subscheme of $X$;
  \item $f$ maps the singular locus of $Y$ birationally onto the non-normal of $X$.
 \end{enumerate}
 A semi-resolution $f\colon Y\rightarrow X$ is called minimal if no $(-1)$-curve is contracted by $f$, that is, there is no exceptional curve $E$ such that $E^2 =K_Y\cdot E = -1$.  
\end{defin}
Semi-resolutions always exist and one can also incorporate a boundary \cite[Sect.~10.5]{KollarSMMP}.

\begin{rem}[Classification of Gorenstein slc singularities]\label{rem: classification of sings}
Semi-log-canonical surface singularities have been classified in terms of their resolution graphs, at least for reduced boundary \cite{ksb88}. 

Let $x\in (X, \Delta)$ be a  Gorenstein slc singularity with minimal log semi-resolution $f\colon Y\to X$.   Then it is one of the following (see \cite[Ch.~4]{Kollar-Mori},  \cite[Sect.~3.3]{KollarSMMP}, and \cite[17]{kollar12}):
\begin{description}[leftmargin=1cm, labelindent=.5cm]
 \item[Gorenstein lc singularities, $\Delta=0$]In this case $x\in X$ is smooth, a
 canonical singularity, or a simple elliptic respectively cusp singularitiy. For the latter the resolution graph is a smooth elliptic curve, a nodal rational curve, or a cycle of smooth rational curves (see also \cite{laufer77} and \cite[Ch.~4]{Reid97}).
\item[Gorenstein lc singularities, $\Delta\neq 0$] Since the boundary is reduced, $\Delta$ has at most nodes. If $\Delta$ is smooth so is $X$ because of the Gorenstein assumption.

 If $\Delta$ has a node at $x$ then $x$ is a smooth point of $X$ or $(X, \Delta)$ is a general hyperplane section of a finite quotient singularity. In the minimal log resolution the dual graph of the exceptional curves is 
\[\bullet   \ {-}\ c_1  \ -\ \cdots  \ - \ c_n  \ {-}
\ \bullet \qquad (c_i\geq1)\]
where $c_i$ represents a smooth rational curve of self-intersection $-c_i$ and each $\bullet$ represents a (local) component of the strict transform of $\Delta$. 
If $c_i=1$ for some $i$ then $n=1$ and $\Delta$ is a normal crossing divisor in a smooth surface.
 \item[non-normal Gorenstein slc singularities, $\Delta=0$]
We describe the dual graph of the $f$-exceptional divisors over $x$: analytically locally $X$ consists of $k$ irreducible components, on each component we have  a resolution graph  as in the previous item, and these are glued together where the components intersect. In total we have a cycle of smooth rational curve. 

For example, the normalisation of the hypersurface singularity $T_{p, \infty, \infty} = \{ xyz+x^p=0\}$ ($p\geq 3$) consists of a plane and an $A_{p-2}$ singularity. The resolution graph of the semi-resolution is obtained by attaching 
\[\bullet   \ {-}\ 2  \ -\ \cdots  \ - \ 2  \ {-}
\ \bullet \text{ and }\bullet   \ {-}\ 1  \ {-}
\ \bullet  \]
to a circle. A more graphical example is given in Figure \ref{fig: semi-resolution}.
\item[non-normal Gorenstein slc singularities, $\Delta\neq 0$]
The difference to the previous case is that the local components are now glued in a chain and the ends of the chain intersect the strict transform of the boundary.
In this case $X$ itself might  not even be $\IQ$-Gorenstein.
\end{description}
\end{rem}

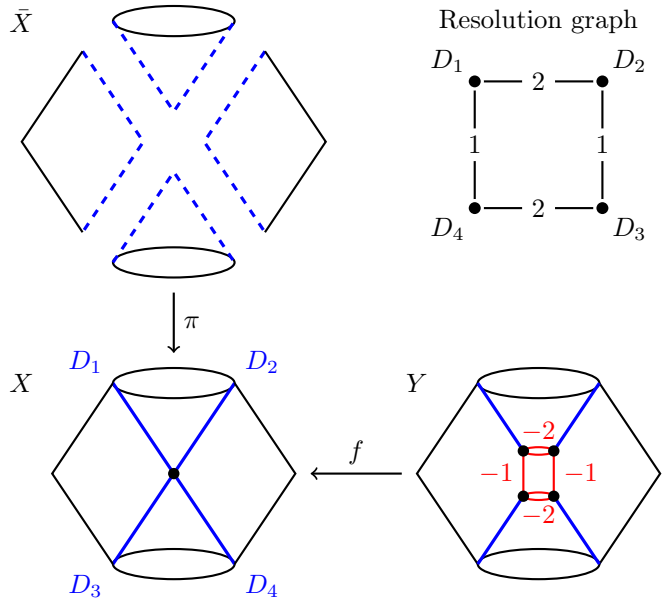
\begin{figure}\caption{Semi-resolution and normalisation of  a degenerate cusp glued from two planes and two $A_1$-singularities. }\label{fig: semi-resolution}
\small
 \begin{tikzpicture}
[thick,
codimtwo/.style = {fill, black, circle, minimum size = .15cm, inner sep = 0pt},
exceptional/.style = {red, thick}, 
nonnormal/.style ={very thick, blue},
boundary/.style={nonnormal, dashed},
boundarycomponent/.style = { fill = black,  minimum size = .2cm, inner sep = 0pt, circle, draw = white,very  thick},
exceptionalcomponent/.style = { fill = white, inner sep = 1.5pt, circle,},
scale = .4]

\begin{scope}[yshift=12cm] % \bar X
\node at (-5, 0) {$\bar X$};
%%% cone 1
 \draw (0,0) circle [x radius=2cm, y radius=.5cm];
\draw[boundary] (-2,0) -- (0, -3)--(2,0);

%% cone 2
\draw (0,-8) circle [x radius=2cm, y radius=.5cm];
\draw[boundary] (-2,-8) -- (0, -5)--(2,-8);

%plane 1
\draw[boundary] (-3,-1) -- ++(2, -3)-- ++ (-2,-3);
\draw (-3,-1) -- ++(-2, -3)-- ++ (2,-3);
%plane 2
\draw[boundary] (3,-1) -- ++(-2, -3)-- ++ (2,-3);
\draw (3,-1) -- ++(2, -3)-- ++ (-2,-3);

\draw[->](0, -9) -- node[right]{$\pi$} ++ (0, -2);
\end{scope}

\begin{scope}%%% X
\node at (-5, 0) {$X$};
 \draw (0,0) circle [x radius=2cm, y radius=.5cm];
\draw (0,-6) circle [x radius=2cm, y radius=.5cm];
\draw[nonnormal] (-2,0) node[above left]{\small $D_1$} -- (2, -6) node[below right]{\small $D_4$};
\draw[nonnormal] (2,0) node[above right]{\small $D_2$} -- (-2, -6) node[below left]{\small $D_3$} ;

\draw (-2,0) -- ++(-2,-3)--++(2,-3);
\draw (2,0) -- ++(2,-3)--++(-2,-3);
\node at (0,-3)[codimtwo] {};
\end{scope}

\begin{scope}[xshift=12cm]%%% Y
\node at (-4, 0) {$Y$};

 \draw (0,0) circle [x radius=2cm, y radius=.5cm];
\draw (0,-6) circle [x radius=2cm, y radius=.5cm];
\draw (-2,0) -- ++(-2,-3)--++(2,-3);
\draw (2,0) -- ++(2,-3)--++(-2,-3);

\draw[nonnormal] (-2,0) -- ++(1.5,-2.25);
\draw[nonnormal] (-2,-6) -- ++(1.5,2.25);
\draw[nonnormal] (2,0) -- ++(-1.5,-2.25);
\draw[nonnormal] (2,-6) -- ++(-1.5,2.25);
 \draw[exceptional] (0,-2.25) circle [x radius=.5cm, y radius=.125cm] node[above] {\small $-2$};
 \draw[exceptional] (0,-3.75) circle [x radius=.5cm, y radius=.125cm] node[below] {\small $-2$};
 \draw[exceptional] (-.5,-2.25) node[codimtwo] {}-- node[left]{\small $-1$} (-.5,-3.75) node[codimtwo] {};
 \draw[exceptional] (.5,-2.25) node[codimtwo] {}-- node[right]{\small $-1$}(.5,-3.75) node[codimtwo] {};

\draw[->](-4.5, -3) -- node[above]{$f$} ++ (-3,0);
\end{scope}

\begin{scope}[xshift = 12cm, yshift=10cm, ]%% resolution graphical
\node at (0,2) {Resolution graph}; 

\draw (0,0) node[exceptionalcomponent] {$2$}
-- ++ (2.1,0) node[boundarycomponent] {}node[above right]{ \small $D_2$}
-- ++ (0,-2.1) node[exceptionalcomponent] {$1$}
-- ++ (0,-2.1) node[boundarycomponent] {}node[below right]{ \small $D_3$}
-- ++ (-2.1,0) node[exceptionalcomponent] {$2$}
-- ++ (-2.1,0) node[boundarycomponent] {}node[below left]{ \small $D_4$}
-- ++ (0,2.1) node[exceptionalcomponent] {$1$}
-- ++ (0,2.1) node[boundarycomponent] {}node[above  left]{ \small $D_1$}
-- ++(2.1,0);
\end{scope}

\end{tikzpicture}
\end{figure}

\subsection{Intersection product and the hat transform.}
Mumford's $\IQ$-valued intersection product can be extended to demi-normal surfaces as long as the curves do not have common components with the conductor, that is, are almost Cartier divisors in the sense of \cite{Hartshorne94}.

\begin{defin}\label{defin: intersection}
 Let $X$ be a demi-normal surface and let  $A$, $B$ be $\IQ$-Weil-divisors on $X$ whose support does not contain any irreducible component of the conductor $D$.  Let $\bar A$ and  $\bar B$ be the strict transforms of $A$ and $B$ on the normalisation $\bar X$. Then we define an intersection pairing
 \[
  AB:= \bar A \bar B
 \]
where the right hand side is Mumford's intersection pairing.
\end{defin}
One should consider the resulting intersection numbers with care if both $A$ and $B$ are not Cartier. For example, the intersection number of curves on different irreducible components of $X$ is always zero in this definition even if the geometric intersection is non-empty.

We also need to recall the hat-transform of a curve in an slc surface constructed in \cite[Appendix A]{liu-rollenske12}.

\begin{prop}\label{prop: hat transform}
Let  $X$ be a demi-normal surface and $B\subset X$ be a curve which does not contain any irreducible component of the conductor $D$.  Let $f\colon Y\rightarrow X$ be the minimal semi-resolution.
\begin{enumerate}
 \item There exists a unique  Weil divisor $\hat B_Y$  on $Y$  such that $f_*\hat B_Y= B$,   for all exceptional divisors $E$ of $f$ we have $\hat B_Y E\leq 0$ and $\hat B_Y$ is minimal with this property.
\item If $X$ has slc singularities then 
\[2p_a(B)-2\leq2 p_a(\hat B_{\bar Y})-2 + {\hat B_{\bar Y} D_{\bar Y}}\]
where $\hat B_{\bar Y}$ is the strict transform of $\hat B_{Y}$ in the normalisation. 
\end{enumerate}
\end{prop}

\section{Riemann--Roch and the $P_2$-inequality}

\begin{theo}[Riemann--Roch for Cartier divisors]\label{thm: RR}
Let $X$ be a demi-normal surface and $L$ be a Cartier divisor on $X$. Then 
\[ \chi(\ko_X(L))=\chi(\ko_X)+\frac1 2 L(L-K_X).\]
\end{theo}
\begin{proof}
 Let $\pi\colon \bar X \to X$ be the normalisation and $\bar D \subset \bar X$ the conductor. 
 We tensor the structure sequence of the double locus,
\begin{equation}\label{eq: sequence of double locus}
 0\to \ki_D\to \ko_X \to \ko_D\to 0,
\end{equation}
with $\ko_X(L)$ and get 
\begin{equation}\label{eq: chi1}
\chi(\ko_X(L)) = \chi( \ki_D\tensor \ko_X(L))  )+\chi(\ko_X(L)\restr D).
\end{equation}
Pulling \eqref{eq: sequence of double locus} back to $\bar X$ and using   projection formula $\pi_*\pi^*\ko_X(L)(-\bar D) = \ki_D\tensor \ko_X(L)$  we have
\[
 \chi(\ki_D\tensor \ko_X(L))=\chi(\pi^*\ko_X(L)) -\chi(\pi^*\ko_X(L)\restr{\bar D})
\]
Adding this to \eqref{eq: chi1} and applying the Riemann--Roch formula for Cartier divisors on normal surfaces (\cite{bla95a}), Riemann--Roch on $D$ and $\bar D$, and Proposition \ref{prop: invariants} we get
\begin{align*}
 \chi(\ko_X(L))&= \chi(\pi^*\ko_X(L))  -\chi(\pi^*\ko_X(L)\restr{\bar D})+ \chi(\ko_X(L)\restr D)\\
& =\chi(\ko_{\bar X}) + \frac{1}{2} \pi^*L(\pi^*L-K_{\bar X})  -\chi(\pi^*\ko_X(L)\restr{\bar D})+ \chi(\ko_X(L)\restr D)\\
&=\chi(\ko_{\bar X}) + \frac{1}{2} \pi^*L(\pi^*L-K_{\bar X} -\bar D +\bar D) \\
&\hspace{4cm}-(\chi(\ko_{\bar D})+\pi^*L\bar D)+ (\chi(\ko_D)+\deg L\restr D) \\
& =\chi(\ko_{\bar X})+\chi(\ko_D)-\chi(\ko_{\bar D})+ \frac{1}{2} L(L-K_X) \\
&=\chi(\ko_X)+\frac1 2 L(L-K_X).\hspace{1cm} \text{(by Proposition~\ref{prop: invariants})}
\end{align*}
This is the claimed formula.
\end{proof}

\begin{cor}\label{cor: plurigenus}
 Let $(X, \Delta)$ be a Gorenstein stable log surface. Then for $m\geq 2$ 
\[
P_m(X, \Delta) = h^0 ( \omega_X(\Delta)^{\tensor m}) =\chi(\ko_X) + \frac{m(m-1)}{2}(K_X+\Delta)^2+ \frac{m}{2}(K_X+\Delta)\Delta
\]
\end{cor}
\begin{proof} 
 Apply Theorem~\ref{thm: RR} to $L=m(K_X+\Delta)$ and use that higher cohomology vanishes by \cite[Corollary~3.4]{liu-rollenske12}. 
\end{proof}

\begin{theo}[$P_2$-inequality]\label{thm: P_2}
Let $(X,\Delta)$ be a Gorenstein stable log surface. Then 
\[ \chi(\omega_X(\Delta))=\chi(\ko_X(-\Delta))\geq -(K_X+\Delta)^2,\]
and equality holds if and only if $\Delta=0$ and $P_2(X)= h^0(X, \omega_X^{\tensor 2})=0$.
\end{theo}
\begin{proof}
If $\Delta = 0$ then  by Corollary \ref{cor: plurigenus} we have
\[0\leq P_2(X) = \chi(\ko_X) + K_X^2,  \]
which gives the claimed formula.

Now suppose $\Delta\neq 0$. We must prove $\chi(\omega_X(\Delta)) + (K_X+\Delta)^2>0$. We begin by applying Theorem~\ref{thm: RR} to $\omega_X(\Delta)$ and $\omega_X(\Delta)^{\tensor 2}$ and taking the difference of the resulting formulas, which gives 
\begin{equation}\label{eq: rr}
\chi(\omega_X(\Delta)) + (K_X+\Delta)^2 = \chi(\omega_X(\Delta)^{\tensor 2})-\frac 1 2 (K_X+\Delta)\Delta.
 \end{equation}

To calculate the right hand side we use the exact sequence 
 \[0\rightarrow \ko_X(2K_X+\Delta) \rightarrow \ko_X(2K_X+2\Delta) \rightarrow \ko_\Delta(2K_X+2\Delta) \rightarrow 0.\]
Together with Riemann--Roch on $\Delta$ we obtain
\[ \chi(\omega_X(\Delta)^{\tensor 2}) = \chi(\ko_X(2K_X+\Delta))+\chi(\ko_\Delta)+2(K_X+\Delta)\Delta,\]
where we used that for a Cartier divisor the degree on a curve coincides with the intersection product. Thus \eqref{eq: rr} becomes
\[\chi(\omega_X(\Delta)) + (K_X+\Delta)^2 = (K_X+\Delta)\Delta+ \chi(\ko_X(2K_X+\Delta))+\left(\chi(\ko_\Delta) +\frac 1 2 (K_X+\Delta)\Delta\right).\]
By \cite[Corollary~3.4]{liu-rollenske12} we have  $H^i(X,2K_X+\Delta) =0$ for $i>0$  and hence 
$  \chi(X,2K_X+\Delta)=h^0(X,2K_X+\Delta)\geq 0.$ Since $K_X+\Delta$ is ample and we assumed $\Delta\neq 0$ the first summand is strictly positive.  It remains to control the last summand. This is accomplished by the following estimate, which seems rather trivial, recall however that we are working on a non-normal surface.

\begin{custom}[Claim] $(K_X+\Delta)\Delta+2\chi(\ko_\Delta)\geq 0$.\end{custom}

\noindent\emph{Proof of the Claim.} Let $f\colon Y\rightarrow X$ be the minimal semi-resolution and $\pi\colon \bar Y\rightarrow Y$ the normalisation. Let $Z$ be the exceptional divisor of $f$ and $\bar Z$ its strict transform in $\bar Y$. Let further $\hat \Delta_{\bar Y}\subset \bar Y$ be the strict transform of the hat transform of $\Delta$ (see Proposition \ref{prop: hat transform})
which, by adjunction on $\bar Y$ satisfies
\begin{equation}\label{eq: boundness of genus of boundary}
 -2\chi(\ko_\Delta) \leq (K_{\bar Y} + D_{\bar Y} + \hat\Delta_{\bar Y})\hat\Delta_{\bar Y}.
\end{equation}

Going through the cases in the classification of Gorenstein slc singularities with non-empty boundary (Remark \ref{rem: classification of sings})  one can compute $\hat \Delta_{\bar Y}$ explicitly. In fact, 
$\hat\Delta_{\bar Y}= \Delta_{\bar Y} + \bar Z'$, where $\bar Z'$ consists of the  reduced connected components of $\bar Z$ that intersect $\Delta_{\bar Y}$.\footnote{To get this simple description it is important to work with the minimal semi-resolution instead of the minimal log semi-resolution. If we  blow up a node of $\Delta$ which is a smooth point of $X$ then the resulting $(-1)$-curve occurs with multiplicity 2 in the hat transform.}
So 
\begin{align*}
 (K_{\bar Y} + D_{\bar Y} + \hat\Delta_{\bar Y})\hat\Delta_{\bar Y}
 &= (K_{\bar Y} + D_{\bar Y} + \Delta_{\bar Y} + \bar Z')\hat\Delta_{\bar Y}\\
 &= (K_{\bar Y} + D_{\bar Y} + \Delta_{\bar Y} + \bar Z)\hat\Delta_{\bar Y} \\
    &= \pi^*f^*(K_X+\Delta)\hat\Delta_{\bar Y}\\
    &= (K_X+\Delta)\Delta.
\end{align*}
Combining this  with \eqref{eq: boundness of genus of boundary} finishes the proof of the claim.
\end{proof}
\begin{rem}
Assume $X$ is a Gorenstein stable surface with $P_2(X)=0$. By the vanishing results in \cite[Sect.\ 2]{liu-rollenske12} we have a surjection $H^0(2K_X)\onto H^0(2K_X\restr D)$, so the latter space has to be zero as well. Thus the degree of $2K_X\restr C$ has to be small, more precisely, by  \cite[Lem.~4.7,~Lem.~4.8]{liu-rollenske12} the non-normal locus cannot be nodal and
\[p_a(D)-1\geq \deg 2K_X\restr{D}\geq 4p_a(D) -4 +2\sum_{p\in\sing D} (2- \mu_p(D)),\]
where $\mu_p(D)$ is the multiplicity of $D$ at $p$. This imposes strong restrictions on the geometry of $D$. 

In fact, the classification results obtained in \cite{fpr13} show  that there is no Gorenstein stable surface with $K_X^2=1$ and $P_2(X)=0$. So we suspect that the $P_2$-inequality might not be sharp. 
\end{rem}

\section{Noether inequality}
In this section we prove the analogue of Noether's inequality for Gorenstein stable log surfaces.

\begin{thm}[Stable log Noether inequality]\label{thm: log noether}
 Let $(X,\Delta)$ be a connected Gorenstein stable log surface (with reduced boundary $\Delta$). Then
  \[p_g(X, \Delta) = h^0(X, \omega_X(\Delta)) \leq (K_X+\Delta)^2 +2.\]
and the inequality is strict if $\Delta=0$.
\end{thm}
\begin{rem}\label{rem: noether sharp}
 The inequality is sharp for pairs, see for example the list of normal log surfaces in Proposition~\ref{prop: normal equality}. We will give a partial characterisation for log surfaces on the stable log Noether line in Corollary \ref{cor: nonnormal equality}.

 For surfaces without boundary the strict inequality is also sharp: there are smooth Horikawa surfaces with $K_X^2=1$ and $p_g(X)=2$.
However, we believe that there are no Gorenstein stable surfaces $X$ such that $p_g(X)-1=K_X^2\geq3$ holds.

If we drop the Gorenstein condition there are normal stable surfaces $X$ such that $p_g(X) > K_X^2 + 2$
\cite[Example~1.8]{TZ92}. The weaker inequality $p_g(X) \leq \ulcorner K_X^2\urcorner + 2$ might be a working hypothesis for the general case.
\end{rem}

\begin{cor}\label{cor: log noether}
 Let $(X,\Delta)$ be a connected Gorenstein stable log surface. Then 
\[ \chi(X,\omega_X(\Delta)) \leq (K_X+\Delta)^2 +2.\]
\end{cor}
\begin{proof}
Note that $\chi(X,\omega_X(\Delta))= h^0(X, \omega_X(\Delta)) - h^1(X, \omega_X(\Delta))+h^2(X, \omega_X(\Delta))$. We have $h^2(X, \omega_X(\Delta))=h^0(X,\ko_X(-\Delta))$ which vanishes if $\Delta \neq0$ and is 1-dimensional if $\Delta =0$. The corollary follows from this and Theorem~\ref{thm: log noether}.
\end{proof}

\subsection{Set-up for the proof}\label{sect: setup noether}
 Consider the minimal semi-resolution $(Y, \Delta_Y)$ of $(X, \Delta)$ and
 the respective normalisations:
\[
\begin{tikzcd}
 D_{\bar Y} \cup\Delta_{\bar Y}\rar[hookrightarrow] \dar& \bar Y \rar{\bar f}\dar{\eta} & \bar X \dar{\pi}\\
D_{ Y} \cup \Delta_Y\rar[hookrightarrow] & Y \rar{f} & X.
\end{tikzcd}
\]
where $\Delta_Y$ (resp.~$\Delta_{\bar Y}$) is the strict transform of $\Delta_X$ in $Y$ (resp.~$\bar Y$).
Our approach is to compute as much as possible on the disjoint union of smooth surfaces $\bar Y$ which we decompose into irreducible components as $\bar Y = \bigsqcup_{i=1}^k\bar Y_i$. Irreducible components of the other spaces involved will be numbered correspondingly, that is, $\bar X_i=\bar f(Y_i)$ and for a divisor $\bar E$ the part contained in $\bar Y_i$ will be denoted by $\bar E_i$.

The non-normal locus of $Y$ is a smooth curve $D_Y$ and the double cover $\eta_D\colon D_{\bar Y}\to D_Y$ corresponds to a line-bundle $\kl$ on $D_Y$ such that $\kl^{\tensor 2}$ is the line bundle associated to the branch points of $\eta_D$. 

A Cartier  divisor $Z$ is defined via $\omega_Y(\Delta_Y + Z) = f^*\omega_X(\Delta)$; it consists of a curve of arithmetic genus 1 for each simple elliptic singularity, cusp or degenerate cusp of $X$ and a chain of rational curves over each the intersection point of $\Delta$ and $D$. The strict transform of $Z$ in the normalisation $\bar Y$ will be denoted by $\bar Z$. With these notations one can check
\begin{gather*} 
\eta^*(K_Y+\Delta_Y+Z) = K_{\bar Y}+ D_{\bar Y}+\Delta_{\bar Y} + \bar Z, \\ \eta^*(K_Y+\Delta_Y + Z)\restr{ D_{\bar Y} }= K_{D_{\bar Y}}+(\Delta_{\bar Y}+\bar Z)\restr{D_{\bar Y}},
\end{gather*}
and
\begin{align}\label{eq: log K^2} 
(K_X+\Delta)^2 &= \eta^*f^*(K_X+\Delta)^2\notag \\
               &= (K_{\bar Y}+D_{\bar Y}+\Delta_{\bar Y}+\bar Z)^2\\
               &= \sum_i (K_{\bar Y_i}+D_{\bar Y, \, i}+\Delta_{\bar Y,\,i} + \bar Z_i)^2.\notag
 \end{align}
To relate $h^0(X,\omega_X(\Delta))$ to the geometry of $\bar Y$ we first note that 
\[h^0(X,\omega_X(\Delta)) =h^0(Y, f^*\omega_X(\Delta))=h^0(Y, \omega_Y(\Delta_Y+Z)) \] by the projection formula.
The exact sequence 
\[
  0\to \ko_{Y}\to \eta_*\ko_{\bar Y} \to \inverse \kl\to 0
\]
tensored with $\omega_Y(\Delta_Y+Z)$ gives the vertical exact sequence:
\[
\begin{tikzcd}
 0 \dar\\
 H^0(X, \omega_X(\Delta)) \dar{(f\circ\eta)^*} \\
 H^0(\bar Y, \omega_{\bar Y}(D_{\bar Y}+\Delta_{\bar Y} + \bar Z))\dar{\rho}\rar{\mathrm{Res}} &  H^0(D_{\bar Y}, \omega_{D_{\bar Y}}(\Delta_{\bar Y} + \bar Z))\dar\\
 H^0(D_Y, \kl^{-1}\otimes\omega_Y(\Delta_Y+Z)\restr{D_Y})
 \rar[equal]  & H^0(D_Y, \kl^{-1}\otimes\omega_Y(\Delta_Y+Z)\restr{D_Y})
\end{tikzcd}
\]

and the map $\rho$ factors through the residue map to the conductor divisor. Thus we have
\begin{equation}\label{eq: log h^0}
H^0(X, \omega_X(\Delta))
= \sum_i h^0(\bar Y_i, \omega_{\bar Y_i}(D_{\bar Y, i}+ \Delta_{\bar Y, i}+ \bar Z_i))- \dim \im \rho.
\end{equation}
In fact,  $\im\rho$ is the quotient of $ H^0(D_{\bar Y}, \omega_{D_{\bar Y}}(\Delta_{\bar Y} + \bar Z))$ by the  $\tau$-invariant subspace by Proposition~\ref{prop: descend section}.

Looking at the equations \eqref{eq: log K^2} and \eqref{eq: log h^0} our claim is trivial unless there is a component $Y_i$ such that  $h^0(\bar Y_i, \omega_{\bar Y_i}(D_{\bar Y, i}+\Delta_{\bar Y,i}+\bar Z_i))$ is bigger than $(K_{\bar Y_i}+D_{\bar Y,  i}+\Delta_{\bar Y,i}+ \bar Z_i)^2$. If there are such components we have to make sure that the image of $\rho$ is big enough, that is, enough sections do not descend from the normalisation to $Y$. This will be done by studying the residue map to single components of $D_{\bar Y}$.

Before adressing these questions in the next subsection we state a Lemma that picks out those components of the boundary is $D_{\bar Y}+\Delta_{\bar Y}+\bar Z $ on the normalisation that (possibly) are contained in the conductor. 
\begin{lem}
Let $\bar C$ be an irreducible component of $D_{\bar Y}+\Delta_{\bar Y} + \bar Z$. Then $\bar C$ is a component of $D_{\bar Y}+\Delta_{\bar Y}$ if and only if $\omega_{\bar Y}(D_{\bar Y}+\bar Z)\restr {\bar C}$ is ample. The line bundle $\omega_{\bar Y}(D_{\bar Y}+\Delta_{\bar Y} + \bar Z)$ restricted on the components of $\bar Z$ is trivial.
\end{lem}

\subsection{Normal pairs with $p_g(W, \Lambda)>(K_W+\Lambda)^2$}
In this section we study normal pairs with large $p_g$ and also lower bounds for the rank of the restriction map  to components of the boundary on which the log-canonical divisor is positive.

We first recall the following facts extracted from \cite{sakai80, TZ92}.
\begin{prop}\label{prop: normal inequality}
 Let $(W, \Lambda)$ be a Gorenstein log canonical pair  with reduced boundary such that $K_W+\Lambda$ is big and nef. Then 
\[ p_g(W,\Lambda)\leq (K_W+\Lambda)^2+2.\]
\end{prop}

\begin{prop}\label{prop: normal equality}
 Let $(W, \Lambda)$ be a Gorenstein log canonical pair such that with reduced boundary such that $K_W+\Lambda$ is ample. Let $f\colon \tilde W\to W$ be the minimal resolution and $\tilde \Lambda = f^*\Lambda$.
 If  $ p_g(W,\Lambda)=(K_W+\Lambda)^2+2$   then the pair $(\tilde W, \tilde \Lambda)$ is one of the following:
\begin{enumerate}
 \item $(\IP^2, \text{ nodal quartic curve})$,
\item $(\IP^2, \text{ nodal quintic curve})$,
\item $\tilde W = \IF_e$, $\tilde \Lambda$ a nodal curve in $|3C_0 +(2e+k+2)F|$ with $k\geq 1$. 
\item $\tilde W = \IF_e$, $\tilde \Lambda$ a nodal curve in $|3C_0 +(2e+2)F|$ with $e\neq 0$.
\end{enumerate}
Here $\IF_e= \IP(\ko_{\IP^1}\oplus \ko_{\IP^1}(-e))\to \IP^1$ denotes a Hirzebruch-surface with section $C_0$ such that $C_0^2 = -e$ and fibre $F$.

A lc pair $(W', \Lambda')$ such that there is a birational morphism $g\colon W'\to W$  with $K_W'+\Lambda'=g^*(K_W+\Lambda)$ big and nef  and $W$ as above is called $p_g$-extremal pair.
\end{prop}
\begin{rem}\label{rem: lin sys}
 Note that if $(W', \Lambda')$ is $p_g$-extremal then $|K_{W'}+\Lambda'|$ defines a morphism whose image is the log-canonical model. In particular, each component of the boundary is either contracted to a point or embedded.
\end{rem}

\begin{lem}\label{lem: log K^2+2} Let $(W', \Lambda')$ be a $p_g$-extremal pair and let $ C_1, C_2$ be components of $\Lambda'$ such that $(K_{W'}+\Lambda')C_i>0$. Then the following holds:
\begin{enumerate}
 \item 
\[r(C_1) = \dim \im \left(\mathrm{Res}\colon H^0(W',  K_{W'}+\Lambda')\to H^0(C_1, (K_{W'}+\Lambda')\restr{C_1})\right)\geq 2.\]
If equality holds then $C_1$ is mapped to a line by the map $\psi$ associated to the linear system $|K_{W'}+\Lambda'|$.
\item If $C_1$ and $C_2$ are two different components then 
\[r(C_1+C_2) = \dim \im \left(\mathrm{Res}\colon H^0(W',  K_{W'}+\Lambda')\to H^0((K_{W'}+\Lambda')\restr{C_1\cup C_2})\right)\geq 3.\]
\end{enumerate}
\end{lem}
\begin{proof} We may assume that $(W', \Lambda')$ is one of the smooth surfaces listed in  Proposition \ref{prop: normal equality}. 

In the first three cases $\psi$ is an embedding while in the fourth case $\psi$ is the minimal resolution of of the cone over a rational normal curve: it contracts the section $C_0$ (see \cite[Ch.\ V]{Hartshorne} for the case of $\IF_e$).  In the latter case,  the section satisfies $(K_{W'}+\Lambda')C_0=0$. Therefore $\psi$ is birational when restricted any of the $C_i$ and $C_1$ and $C_2$ do not have the same image.

Looking only at one component $C_1$ gives $r(C_1)\geq2$ with equality if and only if $C_1$ is mapped to a line, thus \refenum{i}.

Two (different) curves in $\IP^{p_g(W', \Lambda')-1}$ span a projective space of dimension at least two, which gives at least three sections in the restriction, i.e., $r(C_1+C_2)\geq 3$. 
\end{proof}

\begin{lem}\label{lem: log K^2+1}
  Let $(W, \Lambda)$ be a log canonical pair with $W$ smooth and $K_W+\Lambda$ big and nef. If $p_g(W, \Lambda)=(K_W+\Lambda)^2+1$ then for  every  component $C$ of $\Lambda$ such that $(K_W+\Lambda)C>0$ we have
\[r(C) = \dim \im \left(\mathrm{Res}\colon H^0(W, K_W+\Lambda)\to H^0(C, (K_W+\Lambda)\restr C)\right)\geq 1.\]
\end{lem}
\begin{proof}
Without loss of generality we may assume that $W$ does not contain $(-1)$-curves $E$ such that $(K_W+B)E=0$. We write 
\[ |K_{W} + \Lambda| = |H|+G\]
where $G$ is the fixed curve and $|H|$ has at most isolated fixed points; clearly
\[h^0(W, K_W+\Lambda)= h^0(W, H).\]
Our aim is to show that $C$ is not a fixed component of the linear system $|K_{W} + \Lambda|$, that is, $C$ is not contained in $G$. It suffices to show that $K_{W}+\Lambda$ does not have positive degree on any component of $G$, which is equivalent to $(K_{W}+\Lambda)G=0$ because $K_{W}+\Lambda$ is big and nef.
\paragraph{Case 1:} Suppose that the image of the map $\phi$ associated to $|H|$ has dimension two. 
Then both  $H$ and $K_{W}+\Lambda$ are big and nef and thus by assumption
\begin{equation}
 \begin{split}
h^0(W, H)&=h^0(W, K_W+\Lambda)=(K_{W}+\Lambda)^2+1=(H+G)^2+1\\&=H^2+HG+(H+G)G+1\geq H^2+1. 
 \end{split}
\end{equation}
In case of equality we have $HG = (H+G)G=0$ which implies $HG=G^2=0$ and thus $G=0$ by the Hodge-index-theorem and we are done.

So assume that $h^0(W, H)> H^2+1$. Then by  \cite[Lemma~2.1]{TZ92} we have $h^0(W, H)= H^2+2$. Thus 
\[ 1 = (K_{W}+\Lambda)^2- H^2= 2HG+G^2=HG+(K_{W}+\Lambda)G.\]
 If $HG=0$ then $G^2\leq 0$ by the Hodge-index-theorem which gives a contradiction. So $HG=1$ and $(K_{W}+\Lambda)G=0$ and we are done.
\paragraph{Case 2:} 
Suppose that the image of the map $\phi$ associated to $|H|$ has dimension 1.
Our argument follows  closely the proof of \cite[Thm.\ 6.1]{sakai80}. Let $p\colon W^*\to W$ be a minimal resolution of the linear series $|H|$ so that $|p^*H| = |H^*|+E$ and $H^*$ has no base-points.
Denoting by $A$ the image of $\phi$ and by $A^*$ the Stein factorisation we have a diagram
\[
\begin{tikzcd}
 W \rar[dashed]{\phi} & A \rar[hookleftarrow]&  \IP^{h^0(W, K_W+\Lambda)-1}\\
W^*\uar{p} \arrow{ur}\rar{\psi} & A^*\uar{s}
\end{tikzcd}
\]
Then there is a divisor $H$ of degree $n =\deg s\cdot\deg \ko_A(1)\geq h^0(W, K_W+\Lambda)-1$
 on $A^*$ such that $H^*= \psi^*H$.
If $F^*$ is a fibre of $\psi$ and $F$ is its image in $W$ then 
\[ (K_{W}+{\Lambda})^2 = n^2F^2+nFG +(K_{W}+{\Lambda})G\geq n^2F^2+nFG\geq n^2F^2 .\]
Combining all these inequalities we deduce:
\begin{itemize}
 \item If $F^2>0$ then $h^0(W, K_W+\Lambda)-1\leq (K_{W}+{\Lambda})^2$.  Assuming equality we have $(K_{W}+{\Lambda})^2=n=F^2=1$ and $FG = (K_{W}+{\Lambda})G=0$. Consequently $h^0(W, K_W+\Lambda) = 2$ and the fixed part $G=0$.
\item If $F^2 = 0$ then $W=W^*$ and $(K_{W}+{\Lambda})^2 = 2nFG + G^2$. Since $(K_{W}+{\Lambda})^2>0$, the fixed part $G$ is non-empty and $FG>0$. Thus we get 
\begin{align*}
 (K_{W}+{\Lambda})^2 = nFG +(K_{W}+{\Lambda})G \geq n\geq h^0(W, K_W+\Lambda)-1
\end{align*}
Assuming $(K_{W}+{\Lambda})^2=h^0(W, K_W+\Lambda)-1$ we have $(K_{W}+{\Lambda})G=0, FG=1$ and $\deg A=h^0(W, K_W+\Lambda)-1$.
\end{itemize}
So in both cases $K_{W}+{\Lambda}$ has degree zero on every component of the fixed locus of the linear system $|K_{W}+{\Lambda}|$, which concludes the proof.
\end{proof}

\subsection{Proof of Theorem \ref{thm: log noether}} Using \eqref{eq: log K^2} and \eqref{eq: log h^0} we  work on the minimal semi-resolution. If $X$ is irreducible then  the proof of the theorem is a simple corollary of the normal case, so the complexity comes from the fact that we glue several components to one surface. More precisely, we have to control how many sections on the normalisation do not descend to the stable log surface. We will proceed by induction on the number $k$ of irreducible components and in each step define a boundary divisor suited to our purpose. 

To set up the induction we order the  components of $\bar Y$ such that for $1\leq i\leq k$ the surface 
\[U^i:=Y_1\cup\dots \cup Y_i\]
is connected in codimension 1. 
We define a boundary $\Lambda^{i}$ on $U^i$ by the equation
\[ (D_{\bar U^i}+\bar \Lambda^{i})\restr {\bar Y_j}= D_{\bar Y_j}+\Delta_{\bar Y_j} +\bar Z_j \quad (j=1, \dots, i)\] 
where $D_{\bar U^i}$ is the conductor divisor of the normalisation $\bar U^i=\bigsqcup_{j=1}^i\bar Y_j \to U^i$. In other words, we divide the boundary on $\bar Y$ in a part that is the conductor divisor of  $U^i$ and the rest.

Note that since $Y$ is semi-smooth the surface $U^i$ is also a semi-smooth scheme, $(U^i, \Lambda^i)$ is a log surface and  and $(U^k, \Lambda^k) = (Y, \Delta+Z)$.

To conclude we show that each of the pairs $(U^i, \Lambda^i)$ satisfies \[p_g(U^i, \Lambda^i)\leq (K_{U^i}+\Lambda^i)^2+2\] and in the case of equality the following four properties hold for each $1\leq j\leq i$:
\begin{enumerate}
  \item[$(E_1)$]  The pair $(\bar Y_j,D_{\bar Y_j}+\Delta_{\bar Y_j} + \bar Z_j)$ is a $p_g$-extremal pair (see Proposition~\ref{prop: normal equality}).
 \item[$(E_2)$] Each irreducible component of $U^j$ is smooth, that is, each component of the conductor divisor  $D_{U^j}$ is contained in two different irreducible components  of $U^j$.
 \item[$(E_3)$] The linear system $|K_{U^j} + \Lambda^j|$ defines a morphism whose image is the semi-log-canonical model.
 \item[$(E_4)$] The intersection $U_{j-1}\cap Y_j$, called the connecting curve for $U_{j-1}$ and $Y_j$, is a single smooth rational curve $C_j$ that is mapped isomorphically to a line  by the morphism associated to the linear system $|K_{U^j} + \Lambda^j|$ and $ \deg (K_{U^j} + \Lambda^j)\restr C_j = 1$. In particular, $\Lambda^j\neq0$.
\end{enumerate}

\paragraph*{\textbf{Base case of the induction:}}
The surface $U^1=Y_1$ is irreducible and by \eqref{eq: log h^0} and Proposition \ref{prop: normal inequality} we have 
\[ p_g(U^1, \Lambda^1)\leq  p_g(\bar U^1, D_{\bar U^1}+\bar\Lambda^1)\leq (K_{\bar U^1}+D_{\bar U^1}+\bar \Lambda^1)^2+2= (K_{ U^1}+\Lambda^1)^2+2.
\]
If the equality  $p_g(U^1, \Lambda^1)= (K_{ U^1}+\Lambda^1)^2+2$ holds then  the normalisation  $(\bar U^1, D_{\bar U^1}+\bar \Lambda^1)$ is a  $p_g$-extremal pair (see Proposition~\ref{prop: normal equality}); this shows $(E_1)$.

 Moreover,  $\eta^*|K_{U^1}+\Lambda^1|=|K_{\bar U^1}+D_{\bar U^1}+\bar \Lambda^1|$  and hence the map induced by the linear system $|K_{\bar U^1}+D_{\bar U^1}+\bar \Lambda^1|$ on the conductor divisor $D_{\bar U^1}$ factors through the normalisation map and  is  $2:1$. By Remark \ref{rem: lin sys} any component of the boundary of $(\bar U^1, D_{\bar U^1}+\bar \Lambda^1)$ is either contracted or embedded by this linear system. Consequently, the conductor is empty and $U^1$ is normal, which gives $(E_2)$, and also $(E_3)$ again by Remark \ref{rem: lin sys}.

The last property $(E_4)$ is empty for irreducible surfaces, except for the fact that $\Lambda\neq 0$, which is true because the conductor is empty.

\paragraph*{\textbf{Induction step:}}
Now assume $U^{i-1}$ satisfies the inequality and conditions $(E_1)$--$(E_4)$ in case of equality.

Let $C_i=U^{i-1}\cap Y_i$ be the connecting curve. Note that $C_i$ is a  possibly non-connected smooth curve. By Remark~\ref{rem: pushout sections}
\begin{equation}\label{eq: fibre product}
H^0(U^i, K_{U^i}+\Lambda^i) = H^0(U^{i-1}, K_{U^{i-1}}+\Lambda^{i-1})\times_{C_i}H^0(Y_i, K_{Y_i}+\Lambda^i\restr{Y_i}+C_i) 
\end{equation}
where the right hand side is the vector space fibre product induced by the residue maps $r_{U^{i-1}}$ and $r_{Y_i}$ to $C_i$.

Let $r_{\bar Y_i}$ be the residue map to $C_i$ on the normalisation. Pulling back sections on $Y_i$ that vanish along $C_i$ to the normalisation $\bar Y_i$ is injective, so we can estimate
\begin{equation}\label{eq: bound on Y_i}
 \begin{split}
 &h^0(Y_i, K_{Y_i}+\Lambda^i\restr Y_i)- \dim \im (r_{Y_i})\\ 
\leq &h^0(\bar Y_i, K_{\bar Y_i}+ D_{\bar Y_i} + \bar\Lambda^i\restr{\bar Y_i})- \dim \im (r_{\bar Y_i})\\
  \leq &(K_{\bar Y_i}+ D_{\bar Y_i} + \bar\Lambda^i\restr{\bar Y_i})^2\hspace{0.5cm}\text{(Lemmata~\ref{lem: log K^2+2} and \ref{lem: log K^2+1})} \\
  =&(K_{Y_i}+\Lambda^i\restr{Y_i})^2.
 \end{split}
\end{equation}
From equation \eqref{eq: fibre product} we get
\begin{equation}
 \label{eq: final estimate}
\begin{split}
 h^0(U^i,K_{U^i}+\Lambda^i) &\leq h^0(U^{i-1}, K_{U^{i-1}}+\Lambda^{i-1}) + h^0(Y_i, K_{Y_i} + \Lambda^i\restr Y_i+C_i) \\
              & \hspace{4cm}- \max\{\dim\im (r_{U_{i-1}}), \dim\im (r_{Y_i})\}\\
&\leq h^0(U^{i-1}, K_{U^{i-1}}+\Lambda^{i-1}) + h^0(Y_i, K_{Y_i} + \Lambda^i\restr Y_i+C_i)\\
              & \hspace{4cm}-  \dim\im (r_{Y_i})\\
  & \leq  (K_{U^i}+\Lambda^i)^2 +2
\end{split}
\end{equation}
where we have used the induction hypothesis and \eqref{eq: bound on Y_i}. 

It remains to prove that properties $(E_1)$ to $(E_4)$ hold for $U_i$ in the case of equality.
So assume that $h^0(U^i,K_{U^i}+\Lambda^i) = (K_{U^i}+\Lambda^i)^2 +2$. Then all inequalities in \eqref{eq: bound on Y_i} and \eqref{eq: final estimate} are equalities, which implies
\begin{enumerate}
 \item $r_{U^{i-1}}$ and $r_{Y_i}$ have the same image;
  \item $h^0(Y_i, K_{Y_i}+\Lambda^i\restr Y_i+C_i)= (K_{Y_i}+\Lambda^i\restr{Y_i}+C_i)^2+\dim \im (r_{Y_i})$;
  \item $h^0(U^{i-1}, K_{U^{i-1}} )= (K_{U^{i-1}} + \Lambda^{i-1})^2+2$, so $U^{i-1}$ satisfies $(E_1)$--$(E_4)$.
\end{enumerate}
By $(E_3)$, we have $\dim \im (r_{U^{i-1}})\geq 2$ because every component of $C_i$ is embedded by $|K_{U^{i-1}}+\Lambda^{i-1}|$. Thus from \refenum{i} and \refenum{ii} we get
\[ h^0(Y_i, K_{Y_i}+\Lambda^i\restr Y_i+C_i)\geq  (K_{Y_i}+\Lambda^i\restr{Y_i}+C_i)^2+2,\]
so by the base case for the induction equality holds and  the pair  $(Y_i, \Lambda^i\restr{Y_i}+C_i)$ satisfies properties $(E_1)$--$(E_4)$. In particular, $U^i$ satisfies  $(E_1)$.
Moreover,  $\dim \im (r_{Y_i})= 2$ and thus by Lemma \ref{lem: log K^2+2} the connecting curve $C_i$ is isomorphic to $\IP^1$ and $K_{U^i}+\Lambda^i$ has degree 1 on $C_i$, which gives the first part of $(E_4)$ and also that $\Lambda^i\restr Y_i\neq 0$ by the classification of $p_g$-extremal surfaces.  Note also that by dimension reasons both  $r_{U^{i-1}}$ and $ r_{Y_i}$ are surjective.

By \eqref{eq: fibre product} the space  $|K_{U^i}+\Lambda^i|^*$  is spanned by two natural subspaces  $A:=|K_{U^{i-1}}+\Lambda^{i-1}|^*
$ 
and $B:=
| K_{Y_i}+\Lambda^i\restr{Y_i}+C_i|^*
$ and their intersection is  the line $A\cap B = 
|(K_{U^i}+\Lambda^i)\restr C_i|^*
$. 
Thus 
$|K_{U^i}+\Lambda^i|\restr{U^{i-1}}$
embeds $U^{i-1}$ into the subspace $A$ by $|K_{U^{i-1}}+\Lambda^{i-1}|$  and $|K_{U^i}+\Lambda^i|\restr{Y_i}$ embeds $Y_i$ into the subspace $B$ via $|K_{Y_i}+\Lambda^i\restr{Y_i}+C_i|$ such that $C_i$ is embedded as $A\cap B$. In particular, $U^{i-1}$ and $Y_i$ have independent tangent directions along $C_i$ and thus the linear system is an embedding of all of $U_i$ such that $C_i$ is mapped to a line, which proves the second part of $(E_4)$ and $(E_3)$. 
\hfill\qed

The last part of the proof shows the following:
\begin{cor}\label{cor: nonnormal equality}
 Let $(X, \Delta)$ be a Gorenstein stable log surface such that $p_g(X, \Delta)=(K_X+\Delta)^2+2$. Let $X_1, \dots, X_k$ be the irreducible components of $X$. We choose the order such that $V_i=X_1\cup\dots \cup X_i$ is connected in codimension 1. Then
\begin{enumerate}
 \item every irreducible component  $X_i$ is a normal stable log surface as in Proposition \ref{prop: normal equality},
\item the linear system $|K_X+\Delta|$ defines an embedding $\phi\colon X\into \IP=|K_X+\Delta|^*$,
\item $C_i=V_i\cap X_{i+1}$ is a smooth irreducible rational curve,
\item the linear span of $\phi(V_i)$ and of $\phi(X_{i+1})$ intersect exactly in the line $\phi(C_i)$.
\end{enumerate}
In particular, $\Delta\neq 0$.
\end{cor}
We believe that these conditions characterise uniquely Gorenstein stable log surfaces on the stable Noether line.

\section{Examples}
\subsection{Surfaces on a string: covering all invariants}\label{sect: big example}

 We now construct a series of  Gorenstein stable surfaces $X_{k,l}$ with invariants
\[ K_{X_{k,l}}^2 = k\quad \text{and} \quad 1-k \leq \chi(\ko_{X_{k,l}}) = l \leq \begin{cases} k+1 & k\geq 3\\k+2 & k=1,2\end{cases} \]

We want to underline the following consequences of these examples
\begin{enumerate}
 \item The surfaces $X_{k,1}$ have $P_2 =1$ yet $K_{X_{k,1}}^2 = k$ can be arbitrary large.
\item  For every $(a, b) \in \IN \times \IZ$ such that $1-a \leq b\leq a+1$ or $(a,b) \in\{ (1,3), (2,4)\}$ there exists a stable surface $X$ such that $K_X^2 = a$, $\chi(\ko_X) = b$ and with normalisation a disjoint union of projective planes. In particular,  all possible invariants for minimal surfaces of general type are realised by (non-normal) stable surfaces.
\end{enumerate}

\subsubsection{Construction principle}
By Theorem~\ref{thm: triple} a stable surface is uniquely determined by the triple $(\bar X, \bar D, \tau)$ consisting of the normalisation, the conductor divisor and an involution on the normalisation of $\bar D$ preserving the different. To construct $X_{k, \ast}$ we choose:
\begin{description}
 \item[Normalisation] $\bar X = \bigsqcup_{i=1}^k \IP^2$
\item[Conductor] $\bar D$ consists of four general lines in each copy of the plane.
\item[Involution] The normalisation of $\bar D$ is a disjoint union of  $4k$ copies of $\IP^1$ and each copy contains three marked points that map to nodes of $\bar D$. A fixed point free involution $\tau$ preserving the different is uniquely determined by specifying pairs of lines that are interchanged and the action of $\tau$ on the marked points.
\end{description}
The important information is contained in the choice of $\tau$. For simplicity, in each copy of $\IP^2$ we will glue two lines to each other. Different choices for this glueing gives us four different \emph{elementary tiles}, each containing two lines that still have to be glued. 
In a second step we choose $k$ of these elementary tiles and specify how to glue them in a circle to get $X_{k,l}$.

In fact, it will be convenient to work with the minimal semi-resolution 
\[f\colon Y_{k,l}\to X_{k,l},\]
which is more easily visualised and where some computations are more straightforward. So we blow up all intersection points of the four lines as seen in Figure \ref{fig: blow up} and construct a semi-smooth surface $Y_{k,l}$ from $k$ copies of $\tilde\IP^2$ and an involution  $\tilde\tau$ which specifies how to glue the $L_i$ in the various components to each other (preserving the intersection points with the exceptional curves).
\begin{figure}\caption{The basic normal tile.}\label{fig: blow up}
\small
 \begin{tikzpicture}
[curves/.style = {thick},
exceptional/.style = {red, thick}, 
nonnormal/.style ={very thick, blue},
scale =.6]
\node at (0,-4) {$\IP^2$};
\begin{scope}[curves]
 \draw (-.15, 1) -- (1.25, -3);
 \draw (.15, 1) -- (-1.25, -3);
\draw (-1.5, -2.5) -- (1.5, -.5);
\draw (-1.5, -.5) -- (1.5, -2.5);
\end{scope}
\draw[<-] (3, -1.5) to node[above] {\text{blow up}} node[below] {nodes} ++(2,0);

\begin{scope}[xshift = 8cm]
 \node at (0,-4) {$\tilde \IP^2$};
\begin{scope}[curves]
 \draw (-2,1) node[right]{$L_1$} -- (-2, -3);
\draw (2,1)node[right]{$L_2$} -- (2, -3);
\draw (-1.5, -.5) -- ( 1.5, -.5);\node at (0,-.5) [above right] {$L_3$};
\draw (-1.5, -2.5) -- ( 1.5, -2.5); \node at (0,-2.5) [above right] {$L_4$};
\end{scope}
\draw[exceptional] (-2.25, .5) -- (2.25, .5);
\begin{scope}[exceptional, yshift=.5cm]
\draw (0, -.75) -- (0, -3.25);
\draw (-2.25, -.5)  -- ++(-25:1.5cm);
\draw (-2.25, -2.5)  -- ++(-25:1.5cm);
\draw (2.25, -.5)  -- ++(205:1.5cm);
\draw (2.25, -2.5)  -- ++(205:1.5cm);
\end{scope}
\end{scope}
\end{tikzpicture}
\end{figure}

To recover $X_{k,l}$ from $Y_{k,l}$ we just need to contract all exceptional curves. Alternatively one can simply use the same involution in the triple $(\bar X_{k,l}, \bar D, \tau)$.

\begin{rem}\label{rem: sings}
The singularities of the constructed surfaces are very simple to describe (see  \cite{ksb88},  \cite{kollar12} or \cite[Sect.~4.2]{liu-rollenske12}): apart from smooth and normal crossing points we have only very special degenerate cusps. 

Assume  $p$ is  a degenerate cusp on $X_{k,l}$. Then its preimage $\inverse f (p)$  in the semi-resolution $Y_{k,l}$ is a cycle of $m$ $f$-exceptional curves, which become $(-1)$-curves in the normalisation. If $m=1$ then locally analytically $p\in X_{k,l}$ is isomorphic to the cone over a plane nodal cubic, if $m=2$ then locally analytically  $p\in X_{k,l}$ is isomorphic to the origin in $\{x^2+y^2z^2=0\}\subset \IC^3$ (sometimes called $T_{2,\infty, \infty}$), and if $m\geq 3$ then locally analytically $p\in X_{k,l}$ is isomorphic to the  cone over a cycle of $m$ independent lines in projective space.

 By \cite[Sect.~3.4]{stevens98} every such surface is locally smoothable, but usually there are global obstructions.
\end{rem}

\subsubsection{The elementary tiles} 
Consider  $\tilde \IP^2$, the plane blown up in the intersection points of four general lines $L_1, \dots, L_4$. There are six different ways to  glue $L_3$ to $L_4$ while preserving the intersections with the exceptional divisor, which up to isomorphism (renaming $L_1$ and $L_2$) reduce to four essentially different possibilities.  These are given in Figure~\ref{fig: elementary tiles}. The surfaces have normal crossing singularities along $L_{34}$, the image of $L_3$ and $L_4$.
Note that for esthetic reasons we sometimes use a partly mirrored version of Type A in later figures. 

\begin{figure}
\caption{The four different possibilities to glue $L_3$ to $L_4$ up to isomorphism.}
\label{fig: elementary tiles}
\small
\begin{tikzpicture}
[curves/.style = {thick},
exceptional/.style = {red, thick}, 
nonnormal/.style ={very thick, blue},
scale=.6
]

\begin{scope}[curves, xshift = 3cm]%%%%%%%%%% type D
\draw (-2,.5) node[right] {$L_1$}  -- (-2, -4);
\draw (2,.5) node[right] {$L_2$} -- (2, -4);
\draw[nonnormal] (-1.5, -2)-- (-.75, -2) node[below]{$L_{34}$} -- ( 1.5, -2);

\node [draw, rounded corners, thin] at (0, -4.5) {Type D};

\begin{scope}[exceptional]
\clip (-2.25, 1) rectangle (2.25, -5);
\draw  (-2.25, 0) -- (2.25, 0);
\draw[ every loop/.style={looseness=40, min distance=40}]
 (0,-2) ++ (-45: .5cm) to[out=160, in=-60] (0, -2) to[out=120, in =60,loop] () to[out=-120, in=20] ++(225:0.5cm);

\draw ( -1.25, -2)++(-45: .25cm) --   ++(135:2cm);
\draw ( -1.25, -2)++(45: .25cm) --   ++(-135:3cm);

\draw ( 1.25, -2)++(-135: .25cm) --   ++(45:2cm);
\draw ( 1.25, -2)++(135: .25cm) --   ++(-45:3cm);
\end{scope}
\end{scope}

\begin{scope}[curves, xshift=-3cm]%%%%%%%%%% type C
\draw (-2,.5) node[right] {$L_1$}  -- (-2, -4);
\draw (2,.5) node[right] {$L_2$} -- (2, -4);
\draw[nonnormal] (-1.5, -2)-- (-.75, -2) node[below]{$L_{34}$} -- ( 1.5, -2);

\node [draw, rounded corners, thin] at (0, -4.5) {Type C};

\begin{scope}[exceptional]
\clip (-2.25, 1) rectangle (2.25, -5);
\draw  (-2.25, 0) -- (2.25, 0);
\draw[ every loop/.style={looseness=40, min distance=40}]
 (0,-2) ++ (-45: .5cm) to[out=160, in=-60] (0, -2) to[out=120, in =60,loop] () to[out=-120, in=20] ++(225:0.5cm);

\draw ( -1.25, -2)++(-45: .25cm) --   ++(135:2cm);
\draw ( -1.25, -2)++(-135: .25cm) -- ( -1.25, -2) to[out = 45, in = 135, looseness = 2]  (2.25, -1.5);

\draw ( 1.25, -2)++(45: .25cm) -- ( 1.25, -2) to[out = -135, in = -45, looseness = 2]  (-2.25, -2.5);
\draw ( 1.25, -2)++(135: .25cm) --   ++(-45:3cm);
\end{scope}
\end{scope}

\begin{scope}[curves, xshift = -3cm, yshift = 6cm]%%%%%%%%%% type A 
\draw (-2,.5) node[right] {$L_1$}  -- (-2, -4);
\draw (2,.5) node[right] {$L_2$} -- (2, -4);
\draw[nonnormal] (-1.5, -2)-- (-.75, -2) node[above]{$L_{34}$} -- ( 1.5, -2);

\node [draw, rounded corners, thin] at (0, -4.5) {Type A};

\draw[exceptional]  (-2.25, 0) -- (2.25, 0);
\begin{scope}[exceptional, cm ={1,0,0,-1,(0,-4)}]
\clip (-2.25, 1) rectangle (2.25, -5);

\draw ( -1.25, -2)++(-135: .25cm) -- ( -1.25, -2) to[out = 45, in = 135] 
 (0, -2)-- ++(-45: .25cm);
\draw ( 0, -2)++(-135: .25cm) -- (0, -2) to[out = 45, in =135] 
 (2.25, -1.42);

\draw ( -1.25, -2)++(-45: .25cm) --   ++(135:2cm);

\draw ( 1.25, -2)++(135: .25cm) --   ++(-45:3cm);
\draw ( 1.25, -2)++(45: .25cm) -- ( 1.25, -2) to[out = -135, in = -45, looseness = 2]  (-2.25, -2.5);
\end{scope}
\end{scope}

\begin{scope}[curves, xshift = 3cm, yshift = 6cm]%%%%%%%%%% type B
\draw (-2,.5) node[right] {$L_1$}  -- (-2, -4);
\draw (2,.5) node[right] {$L_2$} -- (2, -4);
\draw[nonnormal] (-1.5, -2)-- (.75, -2) node[below]{$L_{34}$} -- ( 1.5, -2);

\node [draw, rounded corners, thin] at (0, -4.5) {Type B};

\begin{scope}[exceptional]
\clip (-2.25, 1) rectangle (2.25, -5);
\draw  (-2.25, 0) -- (2.25, 0);

\draw ( -1.25, -2)++(-135: .25cm) -- ( -1.25, -2) to[out = 45, in = 135] 
 (0, -2)-- ++(-45: .25cm);
\draw ( 0, -2)++(45: .25cm) -- (0, -2) to[out = -135, in = 45] 
 (-2.25, -2.95);

\draw ( -1.25, -2)++(-45: .25cm) --   ++(135:2cm);

\draw ( 1.25, -2)++(-135: .25cm) --   ++(45:2cm);
\draw ( 1.25, -2)++(135: .25cm) --   ++(-45:3cm);

\end{scope}
\end{scope}

\end{tikzpicture}
\end{figure}
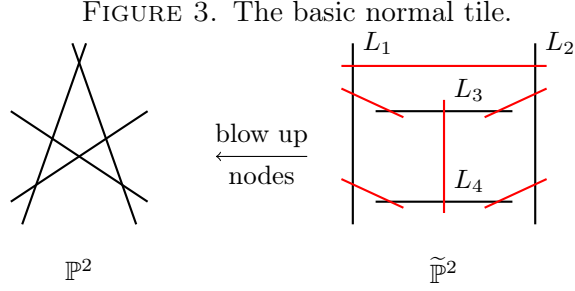

\subsubsection{Warm up: constructing $X_{1,3}$}
As a starting point we describe in detail the surfaces $Y_{1,3}$ and $X_{1,3}$. 

We start with one elementary tile of type C. Then an identification of  $L_1$ and $L_2$ preserving the intersection points with the exceptional divisor is uniquely determined by the images of the intersection points which we indicate with Greek letters. If we contract all exceptional curves we obtain the stable surface $X_{1,3}$. 
\begin{figure}\caption{The surfaces $Y_{1,3}$ and $X_{1,3}$}\label{fig: X_13}
\small
 \begin{tikzpicture}
[curves/.style = {thick},
exceptional/.style = {red, thick}, 
nonnormal/.style ={very thick, blue},
boundary/.style={nonnormal, dashed},
scale = .6]

\begin{scope}[curves]
\draw[boundary] (-2,.5) node[right] {$L_{12}$}  -- (-2, -4);
\draw[boundary] (2,.5) node[right] {$L_{12}$} -- (2, -4);
\draw[nonnormal] (-1.5, -2)-- (-.75, -2) node[below]{$L_{34}$} -- ( 1.5, -2);

\node  at (0, -4.5) {$Y_{1,3}$};

\node[below left] at (-2,0) {$\alpha$};
\node[left] at (-2,-1.3) {$\beta$};
\node[left] at (-2,-2.8) {$\gamma$};
\node[below right] at (2,0) {$\alpha$};
\node[right] at (2,-1.3) {$\beta$};
\node[right] at (2,-2.8) {$\gamma$};

\begin{scope}[exceptional]
\clip (-2.25, 1) rectangle (2.25, -5);
\draw  (-2.25, 0) -- (2.25, 0);
\draw[ every loop/.style={looseness=40, min distance=40}]
 (0,-2) ++ (-45: .5cm) to[out=160, in=-60] (0, -2) to[out=120, in =60,loop] () to[out=-120, in=20] ++(225:0.5cm);

\draw ( -1.25, -2)++(-45: .25cm) --   ++(135:2cm);
\draw ( -1.25, -2)++(-135: .25cm) -- ( -1.25, -2) to[out = 45, in = 135, looseness = 2]  (2.25, -1.5);

\draw ( 1.25, -2)++(45: .25cm) -- ( 1.25, -2) to[out = -135, in = -45, looseness = 2]  (-2.25, -2.5);
\draw ( 1.25, -2)++(135: .25cm) --   ++(-45:3cm);
\end{scope}
\end{scope}

\draw[->] (3, -1.5) to node[above] {resolve} node [below] {deg.\ cusps} ++(2,0);

\begin{scope}[xshift = 7cm, yshift =-1.5cm,  looseness=1.5]
\draw[rounded corners] (-.5, 2) rectangle (4.5, -2);
\node at (2, -3) {$X_{1,3}$};

 \draw[nonnormal, name path = L34] (0,1) to[bend right] (4,.5) node[above] {$L_{34}$};
\draw[boundary, name path = L12] (0,-1) to[bend left] (4,-.5)node[below] {$L_{12}$};
\fill [red, name intersections={of=L12 and L34}]
(intersection-1) circle (3pt) node [above] {$\beta$}
(intersection-2) circle (3pt) node [above]{$\gamma$};
\fill[red] (.3, .75) circle (3pt) node [above right] {$\delta$}
(.3,- .75) circle (3pt) node [below right] {$\alpha$};
\end{scope}
\end{tikzpicture}
\end{figure}
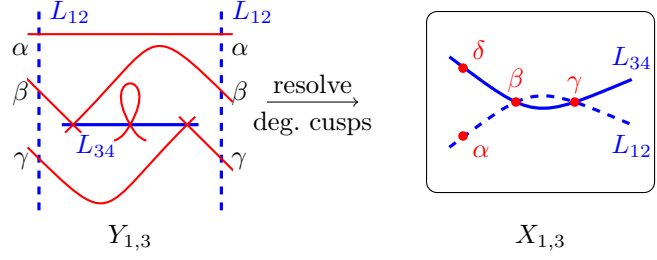
From Figure \ref{fig: X_13} we read off that on $Y_{1,3}$ we have four cycles of exceptional curves, two of length one and two of length two. In $X_{1,3}$ these are contracted to four degenerate cusps.
We will confirm below that $\chi(\ko_{X_{1,3}})=3$ and $K_{X_{1,3}}^2 = 1$.

Either computing the canonical ring directly using \cite[Prop.~5.8]{KollarSMMP} or by reverse engeneering one can check that $X_{1,3}$ is isomorphic to the weighted hypersurface of degree 10 in $\IP(1,1,2,5)$ with equation $z^2+y(x_1^2-y)^2(x_2^2-y)^2$. Geometrically, $X_{1,3}$ is a double cover of the quadric cone in $\IP^3$ branched over the vertex, a plane section and two double plane sections.
From either description we see that $X_{1,3}$ is the degeneration of a smooth Horikawa surface  \cite[VII.(7.1)]{BHPV}.
\subsubsection{Computation of invariants} We now explain how to compute the invariants of surfaces constructed as above.

We get the self-intersection of the canonical divisor by pulling back to the normalisation:
\[ K_{X_{k,l}}^2 =(K_{\bar X_{k,l}}+\text{conductor divisor})^2 =k(K_{\IP^2}+\text{four lines})^2=k.\]

For the holomorphic Euler characteristic we first compute on the semi-resolution $Y_{k,l}$. Let $D_{ Y_{k,l}}$ be the non-normal locus and $D_{\bar Y_{k,l}}$ be the conductor divisor in the normalisation $\bar Y_{k,l}$. Note that $D_{\bar Y_{k,l}}$ is the disjoint union of $4k$ copies of $\IP^1$ and $D_{ Y_{k,l}}$ is the disjoint union of $2k$ copies of $\IP^1$.
Then 
\[ \chi(\ko_{Y_{k,l}}) = \chi(\ko_{\bar Y_{k,l}})-
\chi(\ko_{D_{\bar Y_{k,l}}})+\chi(\ko_{D_{Y_{k,l}}})=
k\chi(\ko_{\tilde \IP^2})+(-4k+2k)\chi(\ko_{\IP^1}) = -k.
\]

Let $c$ be the number of degenerate cusps of $X_{k,l}$ which, by Remark \ref{rem: sings}, corresponds to the number of cycles of exceptional curves in $Y_{k,l}$.  Since by \cite[Lem.~A.7]{liu-rollenske12}  $R^1f_*\ko_{Y_{k,l}}$ is a skyscraper sheaf which has length 1 exactly at the degenerate cusps of $X_{k,l}$  we  have by the Leray spectral sequence and the above computation 
\[\chi(\ko_{X_{k,l}}) = \chi(\ko_{Y_{k,l}})+c = c-k.\]

Going back to $X_{1,3}$ constructed above, we see that there are exactly four degenerate cusps, so $\chi(\ko_{X_{1,3}})=3$ as claimed. 

\begin{rem} It is not very complicated to give a combinatorial formula for  the holomorphic Euler characteristic of $X_{k,l}$ without the use of the semi-resolution and thus avoiding the use of \cite[Lem.~A.7]{liu-rollenske12}.
 \end{rem}

\subsubsection{Construction of $X_{k,l}$ for $1-k\leq l\leq k+1$} The above computations tell us how to proceed: in order for the surface $X_{k,l}$ to have $K_{X_{k,l}}^2=k$ and $\chi(\ko_{X_{k,l}})=l$ we glue $k$ copies of the plane in such a way that the resulting surface has  exactly $k+l$ degenerate cusps. Alternatively, we construct  the semi-resolution $Y_{k,l}$ by glueing  $k$ elementary tiles  such that there are exactly $k+l$ cycles of exceptional curves. 

To construct $Y_{k, 1-k}$ and thus $X_{k,1-k}$ we glue $k$ components of type A in a circle as specified in Figure \ref{fig: Y_k1-k}. There is only one circle of exceptional curves, thus just $1=l+k$ degenerate cusp.
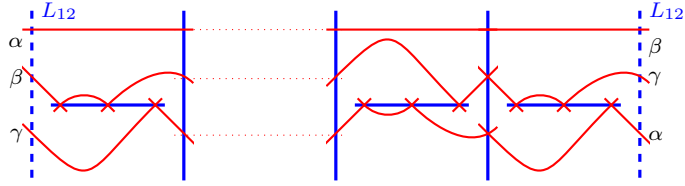
\begin{figure}[h!]\caption{The surface $Y_{k, k-1}$}\label{fig: Y_k1-k}
 \scriptsize
\begin{tikzpicture}
[curves/.style = {thick},
exceptional/.style = {red, thick}, 
nonnormal/.style ={very thick, blue},
boundary/.style={nonnormal, dashed},
scale = .5
]

\begin{scope}[curves]
\draw[boundary] (-10,.5) node[right] {$L_{12}$}  -- (-10, -4);
\draw[boundary] (6,.5) node[right] {$L_{12}$} -- (6, -4);
\foreach \x in {-9.5, -1.5,2.5} {\draw[nonnormal] (\x, -2) -- ++(3,0);};
\foreach \x in {-6, -2,2} {\draw[nonnormal] ( \x, .5) -- ++(0, -4.5);};

\foreach \y in {0, -1.3,-2.8} {\draw[exceptional, thin, dotted] ( -6.25, \y) -- ++(4.5, 0);};

\begin{scope}[xshift=-8cm]
\node[below left] at (-2,0) {$\alpha$};
\node[left] at (-2,-1.3) {$\beta$};
\node[left] at (-2,-2.8) {$\gamma$};
\end{scope}
\begin{scope}[xshift = 4cm]
\node[below right] at (2,0) {$\beta$};
\node[right] at (2,-1.3) {$\gamma$};
\node[right] at (2,-2.8) {$\alpha$};
\end{scope}

\end{scope}

\begin{scope}[exceptional, xshift =-8cm]
\clip (-2.25, 1) rectangle (2.25, -5);
\draw  (-2.25, 0) -- (2.25, 0);

\draw ( -1.25, -2)++(-135: .25cm) -- ( -1.25, -2) to[out = 45, in = 135] 
 (0, -2)-- ++(-45: .25cm);
\draw ( 0, -2)++(-135: .25cm) -- (0, -2) to[out = 45, in =135] 
 (2.25, -1.42);

\draw ( -1.25, -2)++(-45: .25cm) --   ++(135:2cm);

\draw ( 1.25, -2)++(135: .25cm) --   ++(-45:3cm);
\draw ( 1.25, -2)++(45: .25cm) -- ( 1.25, -2) to[out = -135, in = -45, looseness = 2]  (-2.25, -2.5);
\end{scope}
\draw[exceptional]  (-2.25, 0) -- (2.25, 0);
\begin{scope}[exceptional, cm ={1,0,0,-1,(0,-4)}]
\clip (-2.25, 1) rectangle (2.25, -5);

\draw ( -1.25, -2)++(-135: .25cm) -- ( -1.25, -2) to[out = 45, in = 135] 
 (0, -2)-- ++(-45: .25cm);
\draw ( 0, -2)++(-135: .25cm) -- (0, -2) to[out = 45, in =135] 
 (2.25, -1.42);

\draw ( -1.25, -2)++(-45: .25cm) --   ++(135:2cm);

\draw ( 1.25, -2)++(135: .25cm) --   ++(-45:3cm);
\draw ( 1.25, -2)++(45: .25cm) -- ( 1.25, -2) to[out = -135, in = -45, looseness = 2]  (-2.25, -2.5);
\end{scope}
\begin{scope}[exceptional, xshift =4cm]
\clip (-2.25, 1) rectangle (2.25, -5);
\draw  (-2.25, 0) -- (2.25, 0);

\draw ( -1.25, -2)++(-135: .25cm) -- ( -1.25, -2) to[out = 45, in = 135] 
 (0, -2)-- ++(-45: .25cm);
\draw ( 0, -2)++(-135: .25cm) -- (0, -2) to[out = 45, in =135] 
 (2.25, -1.42);

\draw ( -1.25, -2)++(-45: .25cm) --   ++(135:2cm);

\draw ( 1.25, -2)++(135: .25cm) --   ++(-45:3cm);
\draw ( 1.25, -2)++(45: .25cm) -- ( 1.25, -2) to[out = -135, in = -45, looseness = 2]  (-2.25, -2.5);
\end{scope}

\end{tikzpicture}
\end{figure}

 To get $X_{k,l}$ for $1-k<l\leq 1 $ we glue $1-l$ elementary tiles of type A to $k+l-1$ elementary tiles of type B as specified in Figure \ref{fig: chi small}. We read off from the graphical representation that $Y_{k,l}$ contains $c=k+l$ cycles of rational curves and thus $X_{k,l}$ has $k+l$ degenerate cusps.
\begin{figure}[h!]\caption{The surface $Y_{k,l}$ for $1-k<l\leq 1 $}\label{fig: chi small}
\scriptsize
 \begin{tikzpicture}
[curves/.style = {thick},
exceptional/.style = {red, thick}, 
nonnormal/.style ={very thick, blue},
boundary/.style={nonnormal, dashed},
scale = .5
]

\begin{scope}[curves]
\draw[boundary] (-16,.5) node[right] {$L_{12}$}  -- ++(0, -4.5);
\draw[boundary] (8,.5) node[right] {$L_{12}$} -- ++(0, -4.5);
\foreach \x in {-15.5, -9.5,-5.5,  -1.5,4.5} {\draw[nonnormal] (\x, -2) -- ++(3,0);};
\foreach \x in {-12, -10, -6, -2,2, 4} {\draw[nonnormal] ( \x, .5) -- ++(0, -4.5);};
\end{scope}

\begin{scope}[exceptional, thin, dotted]
\foreach \y in {0, -1.3,-2.8} {\draw( 1.75, \y) -- ++(2.5, 0);};
\draw( -12, 0) -- ++(4, 0);
\draw ( -12, -1.3)++(135: .25cm) -- ( -12, -1.3) to[out = -45, in = 45] 
 (-12, -2.8)-- ++(-135: .25cm);
\draw ( -10, -1.3)++(45: .25cm) -- ( -10, -1.3) to[out = -135, in = 135] 
 (-10, -2.8)-- ++(-45: .25cm);
\node[black] at (-11, -2) {\dots}; 
\end{scope}

\begin{scope}[xshift=-14cm]
\node[below left] at (-2,0) {$\alpha$};
\node[left] at (-2,-1.3) {$\beta$};
\node[left] at (-2,-2.8) {$\gamma$};
\end{scope}
\begin{scope}[xshift = 6cm]
\node[below right] at (2,0) {$\alpha$};
\node[right] at (2,-1.3) {$\beta$};
\node[right] at (2,-2.8) {$\gamma$};
\end{scope}

\begin{scope}[exceptional, xshift = -14cm]
\clip (-2.25, 1) rectangle (2.25, -5);
\draw  (-2.25, 0) -- (2.25, 0);

\draw ( -1.25, -2)++(-135: .25cm) -- ( -1.25, -2) to[out = 45, in = 135]  (0, -2)-- ++(-45: .25cm);
\draw ( 0, -2)++(45: .25cm) -- (0, -2) to[out = -135, in = 45]  (-2.25, -2.95);

\draw ( -1.25, -2)++(-45: .25cm) --   ++(135:2cm);

\draw ( 1.25, -2)++(-135: .25cm) --   ++(45:2cm);
\draw ( 1.25, -2)++(135: .25cm) --   ++(-45:3cm);

\end{scope}

\begin{scope}[exceptional, xshift  =-8cm]
\clip (-2.25, 1) rectangle (2.25, -5);
\draw  (-2.25, 0) -- (2.25, 0);

\draw ( -1.25, -2)++(-135: .25cm) -- ( -1.25, -2) to[out = 45, in = 135] 
 (0, -2)-- ++(-45: .25cm);
\draw ( 0, -2)++(45: .25cm) -- (0, -2) to[out = -135, in = 45] 
 (-2.25, -2.95);

\draw ( -1.25, -2)++(-45: .25cm) --   ++(135:2cm);

\draw ( 1.25, -2)++(-135: .25cm) --   ++(45:2cm);
\draw ( 1.25, -2)++(135: .25cm) --   ++(-45:3cm);

\end{scope}
\begin{scope}[exceptional, xshift =-4cm]
\clip (-2.25, 1) rectangle (2.25, -5);
\draw  (-2.25, 0) -- (2.25, 0);

\draw ( -1.25, -2)++(-135: .25cm) -- ( -1.25, -2) to[out = 45, in = 135] 
 (0, -2)-- ++(-45: .25cm);
\draw ( 0, -2)++(45: .25cm) -- (0, -2) to[out = -135, in = 45] 
 (-2.25, -2.95);

\draw ( -1.25, -2)++(-45: .25cm) --   ++(135:2cm);

\draw ( 1.25, -2)++(-135: .25cm) --   ++(45:2cm);
\draw ( 1.25, -2)++(135: .25cm) --   ++(-45:3cm);

\end{scope}

\draw[exceptional]  (-2.25, 0) -- (2.25, 0);
\begin{scope}[exceptional, cm ={1,0,0,-1,(0,-4)}]
\clip (-2.25, 1) rectangle (2.25, -5);

\draw ( -1.25, -2)++(-135: .25cm) -- ( -1.25, -2) to[out = 45, in = 135] 
 (0, -2)-- ++(-45: .25cm);
\draw ( 0, -2)++(-135: .25cm) -- (0, -2) to[out = 45, in =135] 
 (2.25, -1.42);

\draw ( -1.25, -2)++(-45: .25cm) --   ++(135:2cm);

\draw ( 1.25, -2)++(135: .25cm) --   ++(-45:3cm);
\draw ( 1.25, -2)++(45: .25cm) -- ( 1.25, -2) to[out = -135, in = -45, looseness = 2]  (-2.25, -2.5);
\end{scope}
\begin{scope}[exceptional, xshift =6cm]
\clip (-2.25, 1) rectangle (2.25, -5);
\draw  (-2.25, 0) -- (2.25, 0);

\draw ( -1.25, -2)++(-135: .25cm) -- ( -1.25, -2) to[out = 45, in = 135] 
 (0, -2)-- ++(-45: .25cm);
\draw ( 0, -2)++(-135: .25cm) -- (0, -2) to[out = 45, in =135] 
 (2.25, -1.42);

\draw ( -1.25, -2)++(-45: .25cm) --   ++(135:2cm);

\draw ( 1.25, -2)++(135: .25cm) --   ++(-45:3cm);
\draw ( 1.25, -2)++(45: .25cm) -- ( 1.25, -2) to[out = -135, in = -45, looseness = 2]  (-2.25, -2.5);
\end{scope}

\end{tikzpicture}
\end{figure}
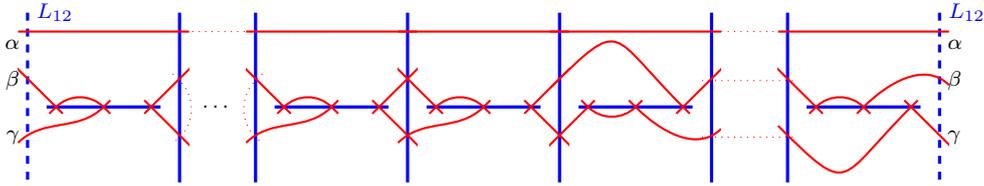

 To get $X_{k,l}$ for $2\leq l\leq k+1 $ we glue $l-1$ elementary tiles of type D to $k-l+1$ elementary tiles of type B as specified in Figure \ref{fig: chi big}. We read off from the graphical representation that $Y_{k,l}$ contains $c=k+l$ cycles of rational curves and thus $X_{k,l}$ has $k+l$ degenerate cusps also in this case.
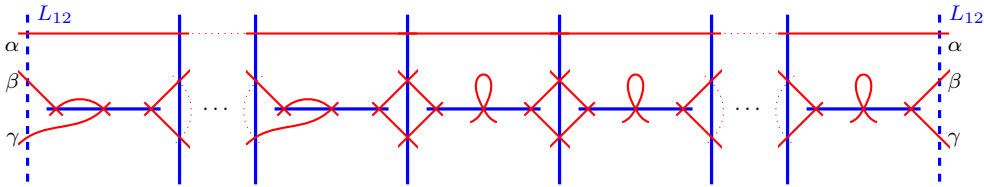
\begin{figure}[h!]\caption{The surface $Y_{k,l}$ for $2\leq l\leq k+1 $}\scriptsize\label{fig: chi big}
 \begin{tikzpicture}
[curves/.style = {thick},
exceptional/.style = {red, thick}, 
nonnormal/.style ={very thick, blue},
boundary/.style={nonnormal, dashed},
scale = .5
]

\begin{scope}[curves]
\draw[boundary] (-16,.5) node[right] {$L_{12}$}  -- ++(0, -4.5);
\draw[boundary] (8,.5) node[right] {$L_{12}$} -- ++(0, -4.5);
\foreach \x in {-15.5, -9.5,-5.5,  -1.5,4.5} {\draw[nonnormal] (\x, -2) -- ++(3,0);};
\foreach \x in {-12, -10, -6, -2,2, 4} {\draw[nonnormal] ( \x, .5) -- ++(0, -4.5);};
\end{scope}

\begin{scope}[exceptional, thin, dotted]
\draw( -12, 0) -- ++(2, 0);
\draw ( -12, -1.3)++(135: .25cm) -- ( -12, -1.3) to[out = -45, in = 45] 
 (-12, -2.8)-- ++(-135: .25cm);
\draw ( -10, -1.3)++(45: .25cm) -- ( -10, -1.3) to[out = -135, in = 135] 
 (-10, -2.8)-- ++(-45: .25cm);
\draw( 2, 0) -- ++(2, 0);
\draw ( 2, -1.3)++(135: .25cm) -- ( 2, -1.3) to[out = -45, in = 45] 
 (2, -2.8)-- ++(-135: .25cm);
\draw ( 4, -1.3)++(45: .25cm) -- ( 4, -1.3) to[out = -135, in = 135] 
 (4, -2.8)-- ++(-45: .25cm);
\node[black] at (-11, -2) {\dots}; 
\node[black] at (3, -2) {\dots}; 

\end{scope}

\begin{scope}[xshift=-14cm]
\node[below left] at (-2,0) {$\alpha$};
\node[left] at (-2,-1.3) {$\beta$};
\node[left] at (-2,-2.8) {$\gamma$};
\end{scope}
\begin{scope}[xshift = 6cm]
\node[below right] at (2,0) {$\alpha$};
\node[right] at (2,-1.3) {$\beta$};
\node[right] at (2,-2.8) {$\gamma$};
\end{scope}

\begin{scope}[exceptional, xshift = -14cm]
\clip (-2.25, 1) rectangle (2.25, -5);
\draw  (-2.25, 0) -- (2.25, 0);

\draw ( -1.25, -2)++(-135: .25cm) -- ( -1.25, -2) to[out = 45, in = 135]  (0, -2)-- ++(-45: .25cm);
\draw ( 0, -2)++(45: .25cm) -- (0, -2) to[out = -135, in = 45]  (-2.25, -2.95);

\draw ( -1.25, -2)++(-45: .25cm) --   ++(135:2cm);

\draw ( 1.25, -2)++(-135: .25cm) --   ++(45:2cm);
\draw ( 1.25, -2)++(135: .25cm) --   ++(-45:3cm);

\end{scope}

\begin{scope}[exceptional, xshift  =-8cm]
\clip (-2.25, 1) rectangle (2.25, -5);
\draw  (-2.25, 0) -- (2.25, 0);

\draw ( -1.25, -2)++(-135: .25cm) -- ( -1.25, -2) to[out = 45, in = 135] 
 (0, -2)-- ++(-45: .25cm);
\draw ( 0, -2)++(45: .25cm) -- (0, -2) to[out = -135, in = 45] 
 (-2.25, -2.95);

\draw ( -1.25, -2)++(-45: .25cm) --   ++(135:2cm);

\draw ( 1.25, -2)++(-135: .25cm) --   ++(45:2cm);
\draw ( 1.25, -2)++(135: .25cm) --   ++(-45:3cm);

\end{scope}
\begin{scope}[exceptional, xshift =-4cm]
\clip (-2.25, 1) rectangle (2.25, -5);
\draw  (-2.25, 0) -- (2.25, 0);
\draw[ every loop/.style={looseness=40, min distance=40}]
 (0,-2) ++ (-45: .5cm) to[out=160, in=-60] (0, -2) to[out=120, in =60,loop] () to[out=-120, in=20] ++(225:0.5cm);

\draw ( -1.25, -2)++(-45: .25cm) --   ++(135:2cm);
\draw ( -1.25, -2)++(45: .25cm) --   ++(-135:3cm);

\draw ( 1.25, -2)++(-135: .25cm) --   ++(45:2cm);
\draw ( 1.25, -2)++(135: .25cm) --   ++(-45:3cm);
\end{scope}

\begin{scope}[exceptional,]
\clip (-2.25, 1) rectangle (2.25, -5);
\draw  (-2.25, 0) -- (2.25, 0);
\draw[ every loop/.style={looseness=40, min distance=40}]
 (0,-2) ++ (-45: .5cm) to[out=160, in=-60] (0, -2) to[out=120, in =60,loop] () to[out=-120, in=20] ++(225:0.5cm);

\draw ( -1.25, -2)++(-45: .25cm) --   ++(135:2cm);
\draw ( -1.25, -2)++(45: .25cm) --   ++(-135:3cm);

\draw ( 1.25, -2)++(-135: .25cm) --   ++(45:2cm);
\draw ( 1.25, -2)++(135: .25cm) --   ++(-45:3cm);
\end{scope}
\begin{scope}[exceptional, xshift =6cm]
\clip (-2.25, 1) rectangle (2.25, -5);
\draw  (-2.25, 0) -- (2.25, 0);
\draw[ every loop/.style={looseness=40, min distance=40}]
 (0,-2) ++ (-45: .5cm) to[out=160, in=-60] (0, -2) to[out=120, in =60,loop] () to[out=-120, in=20] ++(225:0.5cm);

\draw ( -1.25, -2)++(-45: .25cm) --   ++(135:2cm);
\draw ( -1.25, -2)++(45: .25cm) --   ++(-135:3cm);

\draw ( 1.25, -2)++(-135: .25cm) --   ++(45:2cm);
\draw ( 1.25, -2)++(135: .25cm) --   ++(-45:3cm);
\end{scope}

\end{tikzpicture}
\end{figure}

 \subsubsection{The surface $ X_{2,4}$}
The last case cannot be constructed by the same strategy as before. But instead,  to get $X_{2,4}$ we just take two copies of $(\IP^2, \text{nodal quartic curve})$ and let the involution exchange the two curves. The resulting surface has $K_{X_{2,4}}^2 = 2$ and $\chi(\ko_{X_{2,4}}) = 4$; it is a degeneration of a Horikawa surface.

\subsubsection{Smoothability}\label{sect: smoothability}  Locally all constructed surfaces are smoothable by Remark \ref{rem: sings}. Global smoothability is tricky: we have seen above that $X_{1,3}$ is smoothable but on the other hand $X_{9,1}$ cannot be smoothable because minimal surfaces on the Bogomolov-Miyaoka-Yau line are rigid ball quotients (See also Section \ref{sect: example fake}).

\subsubsection{Further variations}
Especially for intermediate values of the invariants there are several other choices of glueing that realise the same invariants. For example, every elementary tile  of type A could be replaced by one of type C thereby increasing the number of degenerate cups and hence $\chi$ by one.  Possibly the resulting surfaces would have different irregularity or geometric genus, but we did not venture into this.

\subsubsection{Non-Gorenstein surfaces}
If we allow the involution $\tau$ to preserve a component of the conductor divisor then it necessarily fixes one of the three marked points. By the classification of slc singularities,  the resulting surface is not Gorenstein but has index two. 

From our building blocks we can also construct non-Gorenstein stable surfaces of index two that violate the stable Noether inequality. To illustrate this we  construct a stable surface  $X_{3, 5}$ with $K_{X_{3,5}}^2=3$ and $\chi(\ko_{X_{3,5}})=5$ given in  Figure \ref{fig: X35}. The two lines $L_1$ and $L_2$ are pinched: on each preimage in the involution one of the marked points is fixed and the other two are exchanged. The two fixed points on these lines give pinch points in the semi-smooth surface $Y_{3,5}$, which are marked by black dots in the picture.

\begin{figure}[ht]\caption{The surface $Y_{3,5}$; the lines $L_1$ and $L_2$ are pinched. }
\label{fig: X35}
\scriptsize
 \begin{tikzpicture}
[curves/.style = {thick},
exceptional/.style = {red, thick}, 
nonnormal/.style ={very thick, blue},
boundary/.style={nonnormal, dashed},
scale = .5
]

\begin{scope}[curves]
\foreach \x in { -9.5,-5.5,  -1.5} {\draw[nonnormal] (\x, -2) -- ++(3,0);};
\foreach \x in {-10,  -6, -2,2} {\draw[nonnormal] ( \x, .5) -- ++(0, -4.5);};
\end{scope}

\begin{scope}[exceptional, xshift  =-8cm]
\clip (-2.25, 1) rectangle (2.25, -5);
\draw  (-2.25, 0) -- (2.25, 0);
\draw[ every loop/.style={looseness=40, min distance=40}]
 (0,-2) ++ (-45: .5cm) to[out=160, in=-60] (0, -2) to[out=120, in =60,loop] () to[out=-120, in=20] ++(225:0.5cm);

\draw ( -2, -2)++(-135: .25cm) -- ( -2, -2) to[out = 45, in = 135]  ++(1, 0)-- ++(-45: .25cm);
\draw ( -2, -2)++(135: .25cm) -- ( -2, -2) to[out = -45, in = -135]  ++(1, 0)-- ++(45: .25cm);

\draw ( 1.25, -2)++(-135: .25cm) --   ++(45:2cm);
\draw ( 1.25, -2)++(135: .25cm) --   ++(-45:3cm);
\end{scope}
\begin{scope}[exceptional, xshift =-4cm]
\clip (-2.25, 1) rectangle (2.25, -5);
\draw  (-2.25, 0) -- (2.25, 0);
\draw[ every loop/.style={looseness=40, min distance=40}]
 (0,-2) ++ (-45: .5cm) to[out=160, in=-60] (0, -2) to[out=120, in =60,loop] () to[out=-120, in=20] ++(225:0.5cm);

\draw ( -1.25, -2)++(-45: .25cm) --   ++(135:2cm);
\draw ( -1.25, -2)++(45: .25cm) --   ++(-135:3cm);

\draw ( 1.25, -2)++(-135: .25cm) --   ++(45:2cm);
\draw ( 1.25, -2)++(135: .25cm) --   ++(-45:3cm);
\end{scope}
\begin{scope}[exceptional]
\clip (-2.25, 1) rectangle (2.25, -5);
\draw  (-2.25, 0) -- (2.25, 0);
\draw[ every loop/.style={looseness=40, min distance=40}]
 (0,-2) ++ (-45: .5cm) to[out=160, in=-60] (0, -2) to[out=120, in =60,loop] () to[out=-120, in=20] ++(225:0.5cm);

\draw ( -1.25, -2)++(-45: .25cm) --   ++(135:2cm);
\draw ( -1.25, -2)++(45: .25cm) --   ++(-135:3cm);

\draw ( 2, -2)++(-45: .25cm) -- ( 2, -2) to[in = 45, out = 135]  ++(-1, 0)-- ++(-135: .25cm);
\draw ( 2, -2)++(45: .25cm) -- ( 2, -2) to[in = -45, out = -135]  ++(-1, 0)-- ++(135: .25cm);
\end{scope}

\fill[black] (-10, 0) circle (3pt) node[above left] {$L_1$};
\fill[black] (2, 0) circle (3pt) node[above right] {$L_2$};

\fill[black] (-10, -3.5) circle (3pt);
\fill[black] (2, -3.5) circle (3pt);

\end{tikzpicture}
\end{figure}
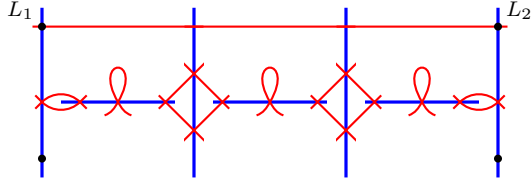

\subsection{Normal geographical examples}\label{sect: exam normal}
 Let $C$ and $C'$ be elliptic curves and let $S = C\times C'$ and fix integers $k,l>0$. Pick general points $P_1, \dots P_k\in C$ and $Q_1, \dots, Q_l\in C'$. Let  
\begin{gather*}
C_i = C\times\{Q_i\}\quad i = 1, \dots, k,\,
 C_j' = \{P_j\}\times C'\quad j = 1, \dots, l.
\end{gather*}

Blowing up the $k\cdot l$ points $(P_i, Q_j)$ in $S$ we get a surface $Y$ with $(-1)$-curves $E_{1,1}, \dots, E_{k,l}$. If $E_i$ (resp.~$E_j'$) is the strict transform of $C_i$ (resp.~$C_j'$) then $E_i$ and $E_j'$ are smooth elliptic curves with 
\[E_i^2 = -l \text{ and } E_j'^2 = -k.\]

We construct a surface $X_{k,l}$ with $k+l$ elliptic singularities by contracting all $E_i$ and $E_j'$:
\[
\begin{tikzcd}
 {}& Y\arrow{dr}{\sigma}\arrow{dl}[swap]{\pi}\\
S && X_{kl}
\end{tikzcd}
\]
The surface $X_{k,l}$ exists as an algebraic space. To prove that it is a stable surface it is enough to show that the Cartier divisor $K_{X_{k,l}}$ is ample, for which we use the Nakai--Moishezon criterion (\cite[Thm.~3.11]{kollar90}).
First note that
\[ K_{X_{k,l}}^2 = \left(K_Y+\sum_i E_i + \sum_j E_j'\right)^2 = \left(\pi^*(\sum_i C_i + \sum_j C_j') -\sum_{i,j} E_{i,j}\right)^2 =kl>0\]
 So let $F$ be an irreducible curve in $X_{k,l}$. Its strict transform $\bar F \subset  Y$ either is one of the $\pi$-exceptional curves or is the strict transform of a curve in $S$. In both cases
\[K_XF = \left(\pi^*(\sum_i C_i + \sum_j C_j')-\sum_{i,j} E_{i,j}\right)\bar F >0\]
and we are done.

To compute $\chi(\ko_{X_{k,l}})$ note that
\[ 0 = \chi(\ko_Y ) = \chi(\ko_{\bar Y}) = \chi(R\sigma_*\ko_{\bar Y}) = \chi(\ko_{X_{k,l}}) -\chi(R^1\sigma_*\ko_{\bar Y}) = \chi(\ko_{X_{k,l}}) -(k+l),\]
because $R^1\sigma_*\ko_{\bar Y}$ has length 1 at each elliptic singular point \cite[Chapter 4]{Reid97}.
Thus we have constructed a normal stable surfaces $X_{k,l}$ such that 
\[\chi(\ko_{X_{k,l}})=k+l\text{ and }  K_{X_{k,l}}^2=kl.\]
In particular, $\chi(\ko_{X_{k,1}})=k+1 =K_{X_{k,l}}^2+1$, which is the ``equality +1'' case of the stable Noether inequality, and $K_{X_{k,k}}^2=k^2> 9\chi(\ko_{X_{k,k}})=18k$ for $k>18$, which confirms that the classical Bogomolov--Miyaoka--Yau inequality does not hold.

\subsection{Further examples}

\subsubsection{Irregularity}
Here we give two examples that show that the irregularity of the normalisation may be larger or smaller than the irregularity of a stable surface.
\begin{exam}[Drop of irregularity]
 Let $(\bar X, \bar D)$ be a principally polarised abelian surface. Then $\bar D$ is a curve of genus two and thus there is a hyperelliptic involution $\tau$ on $\bar D$. 

The  stable surface $X$ correponding to the triple $(\bar X, \bar D, \tau)$ has $q(X)=0$ while $q(\bar X)=2$.
\end{exam}

\begin{exam}[Increase of irregularity]
 In the series of surfaces constructed in Section \ref{sect: big example} we have $\chi(\ko_{X_{1,0}})=0$ so $q(X_{1,0})\geq1$, while on the other hand $q(\bar X_{1,0})=q(\IP^2)=0$.
\end{exam}

\subsubsection{Canonical map}
Here we note some pathologies of the canonical map that make the classical strategy to prove Noether's inequality fail for stable surfaces.
\begin{exam}[The image of the canonical map need not be equidimensional]
  Let $\bar X_1$ be a (smooth) del Pezzo surface of degree 1 and $\bar D$ a nodal curve in $|-2K_{\bar X_1}|$. Pick a smooth surface $\bar X_2$ with the following properties: $\bar D$ is contained in $\bar X_2$ and $\omega_{\bar X_2}(\bar D)$ is very ample and $H^1(\bar X_2, \omega_{\bar X_2} )=0$. In particular $H^0(\omega_{\bar X_2}(\bar D) )$ surjects onto $H^0(\bar D, \omega_{\bar D})$. 

We construct a stable surface $X$ by gluing $\bar X_1$ and $\bar X_2$ along $ \bar D$. By Proposition~\ref{prop: descend section} all sections of $H^0(\bar X_1, \omega_{\bar X_1}(\bar D)) = H^0(\bar X_1, \inverse\omega_{\bar X_1}) $ descend to sections of $\omega_X$. So the image of the canonical map restricted to $ X_1$ is a $\IP^1$ while the image of the canonical map restricted to $X_2$ is a surface.
\end{exam}

\begin{exam}[The image of the canonical map need not be connected]
Let $C$ be a curve of genus at least three which is not hyperelliptic. Glue two copies of $C$ along a point $p\in C$ to get a stable curve $C'$. A straightforward computation shows that the canonical map of $C'$ has a base-point at $p$ and its image is the disjoint union of two copies of $C$ in the canonical embedding.

Now consider the stable surface $X =C'\times C$. As a consequence of the above,  the base locus of the canonical map coincides with the non-normal locus and the image of the canonical map are two copies of $C\times C$ in the canonical embedding. We see that the canonical map is birational while nevertheless its image is not connected. 
\end{exam}

\subsubsection{A family of (fake) fake projective planes}\label{sect: example fake}

 The following example was asked for by Matthias Sch\"utt. It shows explicitly that  we should not expect stable surfaces to exhibit a behaviour similar to smooth surfaces with the same invariants. Concretely, the Gieseker moduli space of surfaces of general type with $K_X^2 =9 $ and $\chi(\ko_X)=1$, whose elements are usually called fake projective planes, consists of isolated points. We will now construct a 1-dimensional family of stable surfaces with the same invariants thus showing that the number of components of the moduli space of stable surfaces goes up and not all stable surfaces with these invariants are rigid.

Let $\bar X_\alpha= \bar X_\beta = \IP^1\times \IP^1$ and $\bar X_\gamma  = \IP^2$ and $\bar X  = \bar X_\alpha\sqcup \bar X_\beta \sqcup \bar X_\gamma$. 
Fix in both $\bar X_\alpha$ and $\bar X_\beta$ the same four horizontal $H_{x,1}, \dots, H_{x, 4}$ and three vertical lines $V_{x,1}, \dots, V_{x,3}$ ($x=\alpha, \beta$) and fix four general lines $L_1, \dots, L_4 \subset \IP^2 = \bar X_\gamma$.

In order to construct a stable surface $X$ we specify an involution $\tau$ on 
\[\bar D^\nu=\bigsqcup_{x = \alpha, \beta}\left(H_{x,1}\sqcup \dots\sqcup H_{x, 4}\sqcup V_{x,1}\sqcup \dots\sqcup V_{x,3} \right)\sqcup L_1\sqcup \dots \sqcup L_4\]
in the following way: first we use the identity $\bar X_\alpha = \bar X_\beta$ to identify
\begin{gather*}
 \tau\colon V_{\alpha,  i } \longleftrightarrow V_{\beta, i} \qquad (i = 1,2,3)\\
\tau\colon H_{\alpha,  i } \longleftrightarrow H_{\beta, i} \qquad (i = 1,2)\\
\end{gather*}
The remaining components all contain 3 marked points for the different and we specify how to glue them by specifying an involution on these points. Points are denoted by the same symbol if they either map to the same node in $\bar X$ in the case of the $L_i$ or if they are identified in the quotient via the gluing of the vertical lines specified above for $H_{\alpha, i}, H_{\beta, i}$. Note that the order of the points is important for result of the gluing.
\begin{gather*}
 \tau\colon H_{\alpha, 3}=\langle a,b,c\rangle\longleftrightarrow \langle 1,2,3\rangle=L_1,\\ 
\tau\colon H_{\beta,  3}=\langle a,b,c\rangle\longleftrightarrow \langle 3,4,5\rangle=L_2,\\
\tau\colon H_{\alpha,  4}=\langle d,e,f\rangle\longleftrightarrow \langle 2,5,6\rangle=L_3,\\
\tau\colon H_{\beta,  4}=\langle d,e,f\rangle\longleftrightarrow \langle 1,6,4\rangle=L_4.
\end{gather*}
 Since \eqref{diagr: pushout} is  a pushout diagram we see that $D$ has 7 singular points: six arise as images of the nodes of $H_{\alpha,1} \cup H_{\alpha,2}\cup V_{\alpha,1}\cup V_{\alpha,2}\cup V_{\alpha,3}$ and the other  is the equivalence classes of the point $a$. The latter has multiplicity 18.
 
Using the normalisation $D^\nu$ we compute $\chi(\ko_D)= 9-17-6=-14$, $\chi(\ko_{\bar D})=-5-5-2=-12$ and consequently, by Proposition \ref{prop: invariants},
\[\chi(\ko_X) = 3+(-14)-(-12)=1,  \quad K_X^2=(K_{\bar X} +\bar D)^2 = 9.\]
Note however that we can vary the cross ratio of the four points in $\IP^1$ corresponding to the 4 horizontal components $H_{x,i}$, thus we have a 1-dimensional family of deformations of $X$.

The surface $X$ is locally smoothable but not globally because all smooth fake projective planes are rigid ball quotients \cite{BHPV}.

\end{document}